\documentclass[10pt]{article}
\usepackage[utf8]{inputenc}
\usepackage[T1]{fontenc}
\usepackage{babel}
\usepackage{amssymb}
\usepackage{enumerate}
\usepackage{amsthm}
\usepackage{amsmath}
\usepackage{xcolor}
\usepackage{enumitem}
\usepackage{hyperref}
\usepackage{geometry}
\usepackage{graphicx}

\usepackage{bm}
\usepackage{bbm}
\usepackage[capitalize]{cleveref}
\usepackage{thmtools}
\usepackage{cases}
\usepackage{ulem}

\usepackage[title]{appendix}

\hypersetup{
    colorlinks=true,
    linkcolor=[rgb]{0.2, 0, 0.7},
    filecolor=magenta,
    urlcolor=cyan,
    bookmarks=true,
    citecolor=[rgb]{0,0.4,0},
}

\crefformat{equation}{#2\normalfont{{(\textup{#1})}}#3}

\theoremstyle{plain}
\newtheorem{thm}{Theorem}[section]

\newtheorem{lem}[thm]{Lemma}

\theoremstyle{definition}
\newtheorem{defi}[thm]{Definition}

\theoremstyle{remark}
\newtheorem{rem}[thm]{Remark}

\crefname{thm}{Theorem}{Theorems}
\crefname{cor}{Corollary}{Corollaries}
\crefname{lem}{Lemma}{Lemmata}
\crefname{prop}{Proposition}{Propositions}
\crefname{defi}{Definition}{Definitions}
\crefname{rem}{Remark}{Remarks}
\crefname{ex}{Example}{Examples}
\numberwithin{equation}{section}
\title{Markov Chain Approximation of Sticky Diffusions and Hamilton–Jacobi–Bellman Equations on Networks}
\author{Alessio Basti\footnote{Dip. di Ingegneria e Geologia, Univ. ``G. d'Annunzio'' di Chieti-Pescara,
		viale Pindaro 42, 65127 Pescara (Italy), {\textsc alessio.basti@unich.it, fabio.camilli@unich.it.}}\and Jules Berry\footnote{ Université Paris-Saclay, CNRS, Centrale-Supélec, Laboratoire des signaux et systèmes, 91190, Gif-sur- Yvette,
		(France){ \textsc jules.berry@centralesupelec.fr}} \and Fabio Camilli$^*$}
\date{\today}
\graphicspath{./}
\begin{document}
\maketitle

\begin{abstract}
	We propose a discrete Markov-chain approximation of diffusion processes on networks with both Kirchhoff  and sticky
vertex conditions. Stickiness is modeled by a probabilistic residence mechanism at the vertex, while the motion along the edges follows an Euler--Maruyama-type update at the diffusive scale. We prove that the associated time-interpolated chain converges in distribution to the limiting diffusion in the Skorokhod space using the Ethier--Kurtz framework.  Based on this construction, we derive a fully discrete semi-Lagrangian scheme for Hamilton--Jacobi--Bellman equations on networks and establish its convergence using viscosity solution techniques.
\end{abstract}

\vskip 4pt

\section{Introduction}
Partial differential equations posed on networks (or metric graphs) have attracted a rapidly growing interest in recent years, motivated by the modeling of systems whose dynamics is constrained to a collection of one-dimensional branches connected through junctions. Typical examples include vehicular traffic and data transmission, but also transport in pipe networks and energy distribution. From the analytical viewpoint, the presence of junctions requires suitable transmission conditions at vertices, and this has motivated a substantial literature for both first- and second-order PDEs on networks, including Hamilton--Jacobi--Bellman (HJB) equations, where well-posedness is
usually established within the viscosity-solution framework (see 
\cite{barles_book}). \par
In parallel, the probabilistic theory of stochastic processes on networks has developed
from the seminal works of Freidlin--Wentzell \cite{fw} and Freidlin--Sheu \cite{fs}, where diffusion processes on graphs naturally arise as limits of classical Markov processes and are characterized by
second-order operators on edges together with vertex conditions.
In particular, in \cite{fw} it is established a general framework for diffusion processes on
graphs via averaging principles, while in \cite{fs} it is derived an
It\^o/Freidlin--Sheu formula on metric graphs and analyzed the role of the local time at vertices.
See also \cite{kps} for a general review.

A key point in the modeling of diffusion-like dynamics on networks is the choice of the
junction condition. The classical Kirchhoff condition can be interpreted
dynamically as a flux balance at the vertex: the outgoing probability fluxes across the
incident edges compensate each other, so that no mass is created or lost at the junction.
In the non-stochastic setting, analogous Kirchhoff-type transmission conditions encode
conservation laws at junctions and ensure a consistent coupling between edges. In recent
years, however, it has become clear that Kirchhoff conditions are often too restrictive for
applications in which the junction induces a delay or a trapping effect. This has
motivated the introduction of more general vertex conditions leading to sticky
behavior: the process has continuous paths but spends a positive amount of Lebesgue time at
the vertex. Sticky vertex conditions naturally arise in network models where junctions act as service
stations or bottlenecks. Typical examples include transport through a hub with a residence
time before rerouting, traffic intersections with waiting effects, and packet networks with
processing delays at routers. In such settings, stickiness provides an effective macroscopic
description of the junction delay while preserving a diffusion-like motion along the edge.
Recent works have developed probabilistic and semigroup approaches to
sticky diffusions on networks; see, e.g.,
\cite{acg,BC2,BC1,bo,bhC,mo,O1,ss,t}.

From the numerical viewpoint, the Euler--Maruyama method is the classical
time-discretization tool to approximate and simulate diffusion processes,
and a rich literature is available on its strong and weak convergence
properties; see, e.g., the monograph of Kloeden--Platen \cite{kp}. 
In stochastic control, Euler--Maruyama-type discretizations have a
particularly natural interpretation: they induce discrete-time controlled
Markov chains whose dynamic programming equations lead to semi-Lagrangian
schemes for the associated HJB equations. This link is well established in
the Euclidean setting, where semi-Lagrangian schemes inherit monotonicity
properties from the underlying controlled Markov chain approximation, and
convergence can be proved by combining consistency, stability, and
comparison principles; see \cite{cf96,ccds}.

The main goal of this work is to extend the classical Markov-chain
approximation approach, inspired by the Euler--Maruyama method, from
Euclidean domains to diffusions on networks with both non-sticky Kirchhoff
and sticky vertex conditions.
We consider a star-shaped network, namely a finite collection of half-lines
glued at a single vertex, which is the canonical local model around a
junction. On each edge, the diffusion is approximated by a nearest-neighbor
random walk at the diffusive scale. More precisely, the spatial mesh is
adapted to the diffusion coefficient on each edge, so that the lattice
points are located at distances \(j\sigma^\iota\sqrt h\)
from the vertex,  where $\sigma^\iota$ is the diffusion coefficient on the edge, $h$ the time step and \(j\in\mathbb N\). This choice is crucial: starting from the first lattice
point on an edge, a downward step lands exactly at the vertex. Hence the
scheme does not overshoot the junction, no truncation of the Euler step is
needed, and no artificial buffer region has to be introduced.

The behavior at the vertex is encoded by a probabilistic residence and
re-emission mechanism. When the chain is at the vertex, it remains there
with a probability depending on the stickiness parameter and is eventually
re-emitted at distance \(\sigma^\iota\sqrt h\) along one of the incident
edges, chosen according to suitable recalibrated weights. In this way, the
scheme reproduces two essential features of the limiting diffusion: the
positive occupation time of the vertex in the sticky case and the correct
redistribution of excursions among the incident edges. In the non-sticky
case, the residence probability degenerates to immediate re-emission, so
that the chain spends only a vanishing amount of physical time at the
junction, consistently with Kirchhoff-type behavior.

For the convergence analysis of the Markov-chain approximation, we adopt
the weak-limit and martingale-problem framework of \cite{ek}. In
particular, tightness and identification of the limit are obtained through
generator-consistency arguments. The edge-adapted construction plays a key
role in this analysis, since the discrete generator admits the appropriate consistency property at the
vertex: uniform consistency in the sticky case and integrated consistency in the
non-sticky case. The overall approach is
therefore naturally aligned with the Markov chain approximation methodology
developed in \cite{kd}, while being tailored to the singular geometry and
transmission conditions of networks. Approximation schemes for sticky diffusions have also been studied in	\cite{ag,bhC}. The space-time Markov-chain approximation developed in
\cite{ag,alv} applies to general finite metric graphs and uses asymmetric
spatial transitions together with state-dependent transition times
	chosen to match the local characteristics of the target diffusion,
	whereas \cite{bhC} considers a sticky random walk on the half-line.
	By comparison, our construction uses a deterministic time step and
	edge-adapted spatial meshes. This choice reproduces
	the edge-dependent diffusion coefficients, ensures that the chain
	reaches the vertex exactly without overshooting or truncation, and
	captures both the sticky residence time and the redistribution of
	excursions among the incident edges.

Building on this discrete framework, our second main contribution is the derivation of a fully discrete
semi-Lagrangian approximation for second-order HJB equations posed on
networks, covering both Kirchhoff and sticky junction conditions. Although
the stochastic approximation is formulated as a Markov-chain scheme, it has
a probabilistic semi-Lagrangian interpretation: the numerical update
follows the local stochastic characteristics of the diffusion, performing
drift-dependent diffusive steps along the edges and implementing the
sticky residence and redistribution mechanism at the vertex.
Semi-Lagrangian schemes for first-order Hamilton--Jacobi equations on
networks have been studied in \cite{cfs,cff}. Here we extend this approach
to second-order HJB equations with diffusion and with both Kirchhoff and
sticky vertex conditions. Moreover, we prove that the fully
discrete solutions converge, as the discretization parameters vanish, to
the unique viscosity solution of the limiting HJB problem on the network.

The paper is organized as follows. In \Cref{sec:prelim} we introduce the
network setting, the relevant function spaces, and the diffusion processes
on \(\Gamma\), including the sticky vertex condition, together with their
generator. In \Cref{sec:euler-Mar} we construct the
Euler--Maruyama-type Markov-chain approximation and prove the weak
convergence of the time-interpolated chain to the limiting diffusion using
the martingale-problem approach. In \Cref{sec:HJB_equation} we derive the
  semi-Lagrangian scheme  for the network
HJB equation, and we establish the  convergence to the viscosity solution
under Kirchhoff and sticky junction conditions.

\section{Networks, Function Spaces and Diffusion Processes}\label{sec:prelim}

Let $\mathcal I=\{1,\dots,N\}$ and let $\{\Gamma^\iota\}_{\iota\in\mathcal I}$ be a family of copies of the half-line $[0,+\infty)$. We define the network $\Gamma$ as the quotient space
\[
\Gamma := \left(\bigsqcup_{\iota\in\mathcal I} \Gamma^\iota\right)\big/\sim,
\]
where the equivalence relation identifies all endpoints, i.e.
\[
(\iota,x)\sim(\kappa,y) \quad \Longleftrightarrow \quad
\begin{cases}
\iota=\kappa,\ x=y,\\
x=y=0.
\end{cases}
\]
The common point is denoted by $O$. Endowed with the distance
\begin{equation}
d\big((\iota,x),(\kappa,y)\big)=
\begin{cases}
|x-y|, & \iota=\kappa,\\
x+y, & \iota\neq\kappa,
\end{cases}
\end{equation}
$\Gamma$ is a locally compact Polish space. 

A function $f:\Gamma\to\mathbb{R}$ is identified with a family
$f=\{f^\iota\}_{\iota\in\mathcal I}$, where $f^\iota(x)=f(\iota,x)$.
We define
\[
C(\Gamma)=
\left\{
f:\ f^\iota\in C([0,+\infty)),\ f^\iota(0)=f^\kappa(0)\ \text{for every $(\iota,\kappa) \in \mathcal{I} \times \mathcal{I}$}
\right\},
\]
\[C_0(\Gamma)
=
\left\{
f\in C(\Gamma):
\lim_{x\to+\infty} f^\iota(x)=0
\quad\text{for every }\iota\in\mathcal I
\right\},
\]
and, for $n\in \mathbb{N}$,
\[
C^n(\Gamma)=
\left\{
f\in C(\Gamma):\ f^\iota\in C^n([0,+\infty))\ \text{for every $\iota \in \mathcal{I}$}
\right\},
\]
where $C^n([0,+\infty))$ denotes the space of functions of class $C^n$ on
$(0,+\infty)$ whose derivatives up to order $n$ extend continuously to
$x=0$; for every $k=1,\ldots,n$, the vertex value
$\partial_x^k f^\iota(0)$ is understood as the corresponding one-sided
limit as $x\downarrow0$..

Throughout the paper, we fix a stickiness parameter $\eta\geq0$ and
positive redistribution weights $(\gamma^\iota)_{\iota\in\mathcal I}$
satisfying
\begin{equation}
	\sum_{\iota\in\mathcal I}\gamma^\iota=1.
\end{equation}
We further assume that, for every $\iota\in\mathcal I$, the following
conditions hold:
\begin{enumerate}
\item[\bf{(H1)}] $b^\iota:[0,+\infty)\to\mathbb R$ is continuous and locally Lipschitz,
with a finite limit $b^\iota(0)$ and $|b^\iota(x)|\le C$ for all $x\ge0$.
\item[\bf{(H2)}] $\sigma^\iota$ is a positive constant on each edge; accordingly we set $\sigma_*:=\min_{\iota\in\mathcal I}\sigma^\iota>0$.
\end{enumerate}
We recall the main definitions and notation for diffusions on
$\Gamma$, following \cite{BC1,fw}.
On each edge $\Gamma^\iota$, $\iota\in \mathcal I$, we consider the second-order differential operator
\[
\mathcal G^\iota f(x)=\frac12(\sigma^\iota)^2 \partial_{xx} f^\iota(x)+b^\iota(x)\partial_x f^\iota(x),
\qquad x>0.
\]
If the limits
$\mathcal G^\iota f(0+)$ coincide for all $\iota\in\mathcal I$, we denote their common value by $\mathcal Gf(O)$ and define
\[
\mathcal Gf(\iota,x)=\mathcal G^\iota f(x),
\qquad x>0,
\]
together with the vertex value $\mathcal Gf(O)$. The sticky diffusion on the network is then characterized by the generator $(\mathcal G,D(\mathcal G))$, where
\[
D(\mathcal G)=
\left\{
f\in C^2(\Gamma)\cap C_0(\Gamma):
\mathcal Gf\in C_0(\Gamma),\quad
\eta \mathcal Gf(O)
=
\sum_{\iota\in\mathcal I}
\gamma^\iota\sigma^\iota \partial_x f^\iota(0)
\right\}.
\]
The parameter $\eta\geq0$ measures the stickiness of the vertex.
When $\eta=0$, the vertex condition reduces to the weighted Kirchhoff condition
\[
\sum_{\iota\in\mathcal I}\gamma^\iota\sigma^\iota \partial_x f^\iota(0)=0,
\]
and the diffusion is non-sticky. When $\eta>0$, the process is sticky: it has continuous paths and
spends a positive amount of Lebesgue time at the vertex $O$. More precisely, if $X(t)=(\iota(t),x(t))$ denotes the diffusion on $\Gamma$, the component
$x(t)$ evolves on each open edge according to the one-dimensional diffusion with drift $b^\iota$ and
variance coefficient $(\sigma^\iota)^2$. At the vertex, the process is delayed by a sticky mechanism and
is then redistributed among the edges according to
\begin{equation}\label{eq:redistribution_prob}
\mathbb P\bigl(\iota(\tau^+)=\iota\mid X_\tau=O\bigr)
=\frac{\gamma^\iota\sigma^\iota}
{\sum_{\kappa\in\mathcal I}\gamma^\kappa\sigma^\kappa},
\qquad \iota\in\mathcal I.
\end{equation}
where \(\tau\) denotes the starting time of the excursion and
\(\iota(\tau^+)\) indicates the edge label immediately after the process
leaves the vertex.

In stochastic differential form, the dynamics is described by the pair
$(X,L)$, where $L=(L(t))_{t\ge0}$ is a continuous, nondecreasing, adapted
process with $L(0)=0$, increasing only on the (random) set
$\{t\ge0:\ X(t)=O\}$, such that
\begin{equation}\label{eq:sticky_SDE}
dx(t)
=
b^{\iota(t)}(x(t))\,\mathbf{1}_{\{X(t)\neq O\}}\,dt
+\sigma^{\iota(t)}\,\mathbf{1}_{\{X(t)\neq O\}}\,dW_t
+\biggl(\sum_{\kappa\in\mathcal I}\gamma^\kappa\sigma^\kappa\biggr)dL(t),
\end{equation}
where $W$ is a one-dimensional Brownian motion, coupled with the
occupation-time (``stickiness'') identity
\begin{equation}\label{eq:occupation_identity}
\int_0^t \mathbf{1}_{\{X(s)=O\}}\,ds=\eta\,L(t),
\qquad t\ge0 .
\end{equation}
The indicator functions in \eqref{eq:sticky_SDE} express
the fact that the standard diffusion dynamics is active only away from the vertex, whereas the
motion at the vertex is governed by the sticky mechanism
\eqref{eq:occupation_identity}. Well-posedness in law of the system
\eqref{eq:sticky_SDE}--\eqref{eq:occupation_identity}, and its equivalence
with the martingale problem for $(\mathcal G,D(\mathcal G))$ follow from
the semigroup construction; see \cite{BC1} and the references therein. The relations \eqref{eq:sticky_SDE}--\eqref{eq:occupation_identity}
do not by themselves determine the edge-label process \(\iota(t)\).
They are therefore complemented by the redistribution rule at the
vertex (see \eqref{eq:redistribution_prob}).

\section{Approximation of the diffusion process on the network: a random-walk scheme}\label{sec:euler-Mar}
In this section, we introduce an edge-adapted Markov-chain
approximation of the diffusion process on a star graph. We first
describe the transition mechanism along the edges and at the vertex,
and then prove the weak convergence of the piecewise-constant
interpolation to the limiting diffusion, treating separately the
sticky case $\eta>0$ and the non-sticky case $\eta=0$.

\subsection{Formulation of the approximation scheme}
For $h>0$, we consider the edge-adapted lattice
\begin{equation*}
	\Gamma_h
	:=
	\{O\}
	\cup
	\bigcup_{\iota\in\mathcal I}
	\left\{
	\bigl(\iota,j\sigma^\iota\sqrt{h}\bigr)
	: j\in\mathbb N,\ j\geq 1
	\right\}
	\subset\Gamma,
\end{equation*}
which constitutes the state space of the approximating chain.
Choose $h_0>0$ such that
\[
\sqrt{h_0}\,\|b\|_\infty\leq\sigma_*,
\]
where $\sigma_*$ is defined in \textbf{(H2)}.
Then, for every $h\in(0,h_0)$, define
\begin{equation}\label{eq:upwind_prob}
	p_\pm^\iota(x)
	:=
	\frac{1}{2}
	\left(
	1\pm\frac{\sqrt{h}}{\sigma^\iota}b^\iota(x)
	\right),
	\qquad
	x\geq0,\quad \iota\in\mathcal I.
\end{equation}
By the choice of $h_0$, these coefficients belong to $[0,1]$ and
therefore define admissible transition probabilities.

Let $(R_n)_{n\in\mathbb N}$, $(I_n)_{n\in\mathbb N}$, and
$(U_n)_{n\in\mathbb N}$ be three mutually independent sequences of
i.i.d.\ random variables with distributions given for each $n\in\mathbb N$ by
\begin{equation}
	\begin{aligned}
		&\mathbb P(R_n=1)
		=\frac{\sqrt h}{\eta+\sqrt h},
		\qquad
		\mathbb P(R_n=0)
		=\frac{\eta}{\eta+\sqrt h},
		\\[4pt]
		&\mathbb P(I_n=\iota)
		=\gamma^\iota,
		\qquad \iota\in\mathcal I,
		\\[4pt]
		&U_n\sim\mathcal U(0,1),
	\end{aligned}
\end{equation}
where $\mathcal U(0,1)$ denotes the uniform distribution on $(0,1)$. 
The approximating chain $X_n^h=(x_n,\iota_n)\in\Gamma_h$ is defined by
\begin{align}
x_{n+1}
&=
\begin{cases}
x_n + \sigma^{\iota_n}\sqrt h\,
\bigl(2\cdot\mathbf 1_{\{U_n\le p_+^{\iota_n}(x_n)\}}-1\bigr),
& x_n>0, \\[4pt]
R_n\,\sigma^{I_n}\sqrt h,
& x_n=0,
\end{cases}
\label{eq:space_fpk}
\\[2pt]
\iota_{n+1}
&=
\begin{cases}
\iota_n,
& x_n>0, \\[4pt]
I_n,
& x_n=0.
\end{cases}
\label{eq:index_fpk}
\end{align}

Here $x=0$ is identified with the vertex $O$. On each open edge, the
chain performs a nearest-neighbor random walk on
$\sigma^\iota\sqrt h\,\mathbb N$, with transition probabilities
\eqref{eq:upwind_prob}. At the vertex, it is held for a geometric
number of steps determined by $\eta$ and then re-emitted at
$\sigma^\iota\sqrt h$ on an edge selected with probability
$\gamma^\iota$. The resulting limiting excursion probabilities are
proportional to $\gamma^\iota\sigma^\iota$, see \eqref{eq:redistribution_prob}.

The chain is locally consistent with the diffusion on each edge: for $x>0$,
\begin{equation}\label{eq:local_consistency}
\mathbb E\bigl[x_{n+1}-x_n\,\big|\,X^h_n=(\iota,x)\bigr]=b^\iota(x)\,h,
\qquad
\operatorname{Var}\bigl(x_{n+1}-x_n\,\big|\,X^h_n=(\iota,x)\bigr)
=(\sigma^\iota)^2h\,\bigl(1+O(h)\bigr),
\end{equation}
in the spirit of the Markov chain approximation method of \cite{kd}. The
crucial structural property, compared with the Euler--Maruyama update, is
that the chain reaches the vertex \emph{exactly}: starting from the first
lattice point $\sigma^\iota\sqrt h$, a downward step lands precisely at $0$.
No buffer region is needed and no truncation of the dynamics ever occurs; as
a consequence, for fixed $\eta>0$, the consistency of the discrete generator holds uniformly over the grid, including the vertex, for sufficiently regular test functions in  $D(\mathcal G)$. In the non-sticky case, only an integrated consistency property is available.

\begin{rem}[Calibration of the vertex mechanism]\label{rem:calibration}
For $b\equiv0$, starting from $\sigma^\iota\sqrt h$, the chain returns to
$O$ before reaching $\rho>0$ with probability
$1-\sigma^\iota\sqrt h/\rho+o(\sqrt h)$. Hence the expected number $N_\rho^h$ of vertex sojourns before leaving the ball of radius $\rho$ satisfies
\[
\mathbb E[N_\rho^h]
=
\frac{\rho}
{\sqrt h\sum_{\kappa\in\mathcal I}
	\gamma^\kappa\sigma^\kappa}
+O(1),
\]
and the probability of exiting through edge $\iota$ is \eqref{eq:redistribution_prob}.
Each sojourn has expected duration $\eta\sqrt h+h$, so the expected
occupation time before exit is
\[
\frac{\eta\rho}{\sum_{\kappa\in\mathcal I}
\gamma^\kappa\sigma^\kappa}+O(\sqrt h).
\]
When $\eta=0$, the chain is immediately re-emitted from the vertex.
\end{rem}

We associate to the Markov chain $(X_n^h)_{n\geq0}$ a continuous-time
process $\widehat X^h=(\widehat X^h(t))_{t\geq0}$ defined by piecewise
constant interpolation:
\begin{equation}\label{eq:interpolation}
\widehat X^h(t)
:=
X^h_{\lfloor t/h\rfloor}
=
\bigl(\iota^h(t),x^h(t)\bigr),
\qquad t\geq0 .
\end{equation}
Since $\Gamma_h\subset\Gamma$, the process $\widehat X^h$ is a piecewise
constant process, with values in $\Gamma$, and
therefore $\widehat X^h\in D([0,+\infty);\Gamma)$. 
On $\Gamma_h$ we define the discrete generator $\mathcal G_h$ by
\begin{align*}
\mathcal G_h\varphi(\iota,x)
&:=
\frac{\mathbb E_{\iota,x}\!\left[\varphi(X_1^h)\right]-\varphi(\iota,x)}{h}\\
&=\begin{cases}
	\displaystyle
	\frac{
		p_+^\iota(x)\,\varphi\!\left(\iota,x+\sigma^\iota\sqrt{h}\right)
		+
		p_-^\iota(x)\,\varphi\!\left(\iota,x-\sigma^\iota\sqrt{h}\right)
		-
		\varphi(\iota,x)
	}{h},
	& x>0,
	\\[2.2em]
	\displaystyle
	\frac{1}{\sqrt{h}\,(\eta+\sqrt{h})}
	\left[
	\sum_{\kappa\in\mathcal I}
	\gamma^\kappa\,
	\varphi\!\left(\kappa,\sigma^\kappa\sqrt{h}\right)
	-
	\varphi(O)
	\right],
	& x=0.
\end{cases}
\end{align*}
for every bounded $\varphi:\Gamma_h\to\mathbb R$. For $\varphi\in D(\mathcal G)$
we simply write $\mathcal G_h\varphi$ for $\mathcal G_h$ applied to the
restriction $\varphi|_{\Gamma_h}$.

\subsection{Convergence of the scheme}

We can now state the convergence result. Its proof relies on the classical
Ethier--Kurtz approach \cite[Chapter 3]{ek} (see also \cite{kp,kd}): compact
containment from a Lyapunov estimate, tightness via Aldous' criterion, and
identification of the limit through the uniform convergence of the discrete
generators and the martingale problem. Uniqueness then yields convergence in
distribution.

\begin{thm}\label{thm:EK_convergence}
Assume $\eta>0$. Let $\widehat X^h$ be the piecewise constant
interpolation \eqref{eq:interpolation} of the Markov chain
\eqref{eq:space_fpk}--\eqref{eq:index_fpk}, and assume that
$X_0^h\to x\in\Gamma$ as $h\to0$. Then, for every $T>0$,
\[
\widehat X^h \Rightarrow X
\qquad\text{in }D([0,T];\Gamma),
\]
where $X$ is the sticky diffusion on $\Gamma$ with generator
$(\mathcal G,D(\mathcal G))$.
\end{thm}

We start with
the Lyapunov estimate, whose proof is postponed to Appendix~\ref{app:proof_Euler}.

\begin{lem}\label{lem:lyapunov}
Let $V:\Gamma\to \mathbb R$ be defined by $V(\iota,x)=1+x^2$.
Then there exist $h_0>0$ and $C>0$, independent of $h\in(0,h_0)$, such that
\begin{equation}\label{lyap_est}
\mathcal G_h V(\iota,x)\le C V(\iota,x),
\qquad
(\iota,x)\in\Gamma_h .
\end{equation}
Moreover, if $X_0^h\to x\in\Gamma$, then for every $T>0$ and every
$\varepsilon>0$ there exists a compact set $K\subset\Gamma$ such that
\begin{equation}\label{compact}
\inf_{0<h<h_0}
\mathbb P\left(
\widehat X^h(t)\in K,\ \forall t\in[0,T]
\right)
\ge 1-\varepsilon .
\end{equation}
\end{lem}

\begin{lem}\label{lem:tightness}
Assume $\eta>0$. The family $\{\widehat X^h\}_{h>0}$ is tight in
$D([0,T];\Gamma)$.
\end{lem}

Having established tightness, it remains to identify the law of any weak
limit through the martingale problem associated with the sticky generator
$(\mathcal G,D(\mathcal G))$. The key point is the following uniform
consistency estimate, whose proof is postponed to
Appendix~\ref{app:proof_Euler}.

\begin{lem}\label{lem:local-generator-estimates}
Assume $\eta>0$. Let $\varphi\in D(\mathcal G)$ with
$\varphi^\iota\in C_b^3([0,+\infty))$ for every $\iota\in\mathcal I$.
Then there exist $h_0>0$ and $C>0$, depending only on
$\|\varphi\|_{C^3}$, on the coefficients and on $\eta$, such that for all
$h\in(0,h_0)$:
\begin{enumerate}
\item[\rm(i)] for every $(\iota,x)\in\Gamma_h$ with $x>0$,
\[
\left|
\mathcal G_h\varphi(\iota,x)-\mathcal G\varphi(\iota,x)
\right|
\le C\sqrt h \, ;
\]
\item[\rm(ii)] at the vertex,
\[
\left|
\mathcal G_h\varphi(O)-\mathcal G\varphi(O)
\right|
\le \frac{C}{\eta}\,\sqrt h .
\]
\end{enumerate}
In particular,
$\displaystyle\sup_{y\in\Gamma_h}|\mathcal G_h\varphi(y)-\mathcal G\varphi(y)|
\le C\sqrt h\to0$.
\end{lem}

\begin{proof}[Proof of \Cref{thm:EK_convergence}]
Let $\varphi\in D(\mathcal G)$ be such that 
$\varphi^\iota\in C_b^3([0,+\infty))$ for every $\iota\in\mathcal I$.
The class of such functions is a core for $\mathcal G$, equivalently, it is dense in $D(\mathcal G)$ with respect to the graph norm, and it
therefore determines the associated martingale problem; see \cite{BC1}. For the discrete
chain,
\[
M_n^{h,\varphi}
:=
\varphi(X_n^h)-\varphi(X_0^h)
-
h\sum_{k=0}^{n-1}\mathcal G_h\varphi(X_k^h)
\]
is a martingale, and so is its piecewise constant interpolation
\[
\widehat M_t^{h,\varphi}
:=
\varphi(\widehat X_t^h)-\varphi(\widehat X_0^h)
-
\int_0^{\lfloor t/h\rfloor h}
\mathcal G_h\varphi(\widehat X_s^h)\,ds
\]
with respect to the natural filtration of $\widehat X^h$.
By \Cref{lem:tightness}, every sequence $h\downarrow0$ admits a
subsequence, still denoted by $h$, such that
\[
\widehat X^h\Rightarrow X
\qquad\text{in }D([0,T];\Gamma).
\]
Moreover, every jump of $\widehat X^h$ has size bounded by
\[
\sup_{0\le t\le T}
d\bigl(\widehat X^h(t),\widehat X^h(t-)\bigr)
\leq
\|\sigma\|_\infty\sqrt h.
\]
Indeed, an interior transition has radial size
$\sigma^\iota\sqrt h$, while a transition from the vertex either leaves
the process at $O$ or moves it to the first grid point of one of the
edges. Consequently,
\[
\sup_{0\le t\le T}
d\bigl(\widehat X^h(t),\widehat X^h(t-)\bigr)
\longrightarrow0
\]
deterministically as $h\to0$. Hence the tight family
$\{\widehat X^h\}_{h>0}$ is in fact $C$-tight, and every weak limit has
continuous paths. In particular,
\[
\mathbb P\bigl(X\in C([0,T];\Gamma)\bigr)=1.
\]

We now prove that every such limit solves the martingale problem for
$(\mathcal G,D(\mathcal G))$. Fix $0\leq s<t\leq T$ and let $F$ be a
bounded continuous functional on $D([0,s];\Gamma)$. Since
$\widehat M^{h,\varphi}$ is a martingale,
\begin{equation}\label{eq:discrete_mg_identity}
	\mathbb E\left[
	F\bigl(\widehat X^h_{\cdot\wedge s}\bigr)
	\left(
	\widehat M_t^{h,\varphi}
	-
	\widehat M_s^{h,\varphi}
	\right)
	\right]=0.
\end{equation}
By \Cref{lem:local-generator-estimates},
\[
\begin{split}
	\left|
	\int_s^t
	\mathcal G_h\varphi(\widehat X_u^h)\,du
	-
	\int_s^t
	\mathcal G\varphi(\widehat X_u^h)\,du
	\right|
	&\leq
	T\sup_{y\in\Gamma_h}
	\left|
	\mathcal G_h\varphi(y)-\mathcal G\varphi(y)
	\right| \\
	&\leq C T\sqrt h
	\longrightarrow0.
\end{split}
\]
The errors caused by replacing
$\lfloor t/h\rfloor h$ and $\lfloor s/h\rfloor h$ by $t$ and $s$ are
bounded by $2h\|\mathcal G_h\varphi\|_\infty$
which tends to zero because the uniform generator consistency implies
that
\[
\sup_{0<h<h_0}\|\mathcal G_h\varphi\|_\infty<+\infty.
\]

Since every weak limit has continuous paths, convergence in the
Skorokhod $J_1$ topology to $X$ implies locally uniform convergence in
time along any Skorokhod representation. As $\varphi$ and
$\mathcal G\varphi$ are bounded and continuous on $\Gamma$, the
continuous mapping theorem therefore yields
\[
	 \varphi(\widehat X_t^h)-\varphi(\widehat X_s^h)
	-\int_s^t\mathcal G\varphi(\widehat X_u^h)\,du\Rightarrow
	\varphi(X_t)-\varphi(X_s)
	-\int_s^t\mathcal G\varphi(X_u)\,du,
\]
jointly with
\[
F\bigl(\widehat X^h_{\cdot\wedge s}\bigr)
\Rightarrow
F\bigl(X_{\cdot\wedge s}\bigr).
\]

All the random variables involved are uniformly bounded. Therefore,
passing to the limit in \eqref{eq:discrete_mg_identity} gives, for every
$0\leq s<t\leq T$,
\[
\mathbb E\left[
F(X_{\cdot\wedge s})
\left(
M_t^\varphi-M_s^\varphi
\right)
\right]=0,
\]
where
\[
M_t^\varphi
:=
\varphi(X_t)-\varphi(X_0)
-
\int_0^t\mathcal G\varphi(X_u)\,du.
\]
Since the identity holds for every bounded continuous functional $F$
depending only on the path up to time $s$, a standard monotone-class
argument shows that $M^\varphi$ is a martingale with respect to the
natural filtration of $X$ and every weak limit of $\widehat X^h$ solves
the martingale problem for $(\mathcal G,D(\mathcal G))$. By the well-posedness of the martingale problem for
$(\mathcal G,D(\mathcal G))$, established in \cite{BC1}, every weak
limit has the same law. The standard Ethier--Kurtz convergence theorem
then implies convergence of the whole family.
\end{proof}
\begin{rem}[Extension to graphs with finite edges]
	\label{rem:finite_edges}
	The restriction to an infinite star graph allows us to focus on the
	local behavior at the vertex without introducing additional
	mesh-compatibility conditions. The edge-adapted construction extends
	directly to a finite metric graph when the intrinsic edge lengths are
	commensurable. More precisely, let $e_i$ be an edge of length $L_i>0$ with constant
	diffusion coefficient $\sigma_i>0$. In order that the lattice on
	$e_i$ reach both endpoints exactly, there must exist
	$J_i\in\mathbb N\setminus\{0\}$ such that $	J_i\sigma_i\sqrt h=L_i$.
	Hence, a common time step $h$ can be used on all the edges only if
	the intrinsic lengths $ \ell_i:=L_i/\sigma_i$
	are commensurable. Under this assumption, one can choose a
	sequence of compatible time steps converging to zero, and the
	construction and the local consistency arguments developed above
	apply at every vertex.
\end{rem}

\begin{rem}[Space-dependent diffusion coefficients]
	For uniformly positive and sufficiently regular space-dependent
	diffusion coefficients, the construction may be formulated in the
	intrinsic coordinates
	\[
	s_\iota(x)
	:=
	\int_0^x\frac{dr}{\sigma_\iota(r)}.
	\]
	The diffusion coefficient then becomes constant, while the drift is
	replaced by
	\[
	\beta_\iota(y)
	=
	\frac{b_\iota(s_\iota^{-1}(y))}
	{\sigma_\iota(s_\iota^{-1}(y))}
	-
	\frac12
	\sigma_\iota'(s_\iota^{-1}(y)).
	\]
	The chain can therefore be constructed on the uniform lattice
	$\sqrt h\,\mathbb N$, retaining the exact hitting of the vertex.
	Extending the convergence results to this setting would require
	additional estimates and is not pursued here.
\end{rem}


\subsection{The non-sticky case  $\eta=0$}

We briefly discuss the approximation scheme in the non-sticky case
$\eta=0$. In this case, the vertex condition in the domain of the
generator reduces to the weighted Kirchhoff condition
\begin{equation*}
	\sum_{\iota\in\mathcal I}
	\gamma^\iota \sigma^\iota
	\partial_x f^\iota(0)=0.	
\end{equation*}
The limiting process is therefore the non-sticky Walsh-type diffusion
on $\Gamma$: it has continuous paths, spends zero Lebesgue time at the
vertex, and, whenever a new excursion starts from $O$, the outgoing
edge is selected according to the redistribution probabilities given by \eqref{eq:redistribution_prob}.
The edge-adapted lattice and the transition mechanism on the open
edges are unchanged. The only difference concerns the vertex: when
$\eta=0$, one has $R_n=1$ almost surely, so that the chain is
immediately re-emitted from $O$. More precisely, if $X_n^h=O$, then
\begin{equation}
	X_{n+1}^h
	=
	\bigl(I_n,\sigma^{I_n}\sqrt{h}\bigr),
	\qquad
	\mathbb{P}(I_n=\iota)=\gamma^\iota.
	\label{eq:nonsticky-vertex-transition}
\end{equation}
Thus, no geometric holding mechanism is present at the vertex, and
the occupation time associated with the discrete visits to $O$ will be shown to
vanish as $h\to0$.

The discrete generator is unchanged at the interior grid points,
whereas at the vertex it takes the form
\begin{equation}
	\mathcal G_h\phi(O)
	=
	\frac{1}{h}
	\sum_{\iota\in\mathcal I}
	\gamma^\iota
	\left[
	\phi^\iota(\sigma^\iota\sqrt{h})-\phi(O)
	\right].
	\label{eq:nonsticky-discrete-generator}
\end{equation}

The convergence proof for $\eta=0$ differs from the sticky case only
in the treatment of the vertex. More precisely, the following two
points require a separate argument:
\begin{itemize}
	\item
	Tightness. The proof of Lemma~\ref{lem:tightness} for
	$\eta>0$ relies on the estimate
	\[
	\left|
	\mathbb{E}
	\left[
	x_{n+1}-x_n\mid X_n^h=O
	\right]
	\right|
	\leq Ch,
	\]
	where the constant depends on $1/\eta$. When $\eta=0$, the chain
	is immediately re-emitted from the vertex and the corresponding
	increment is of order $\sqrt{h}$. Therefore, tightness requires a
	separate control of the number of visits to the vertex, or,
	equivalently, of the discrete occupation time at $O$.
	
	\item
	Identification of the limit. The uniform consistency
	estimate of Lemma~\ref{lem:local-generator-estimates} does not extend
	directly to the vertex. Indeed,  expanding
	\eqref{eq:nonsticky-discrete-generator} gives
	\begin{equation}
		\begin{split}
		\mathcal G_h\phi(O)
		&=
		\frac{1}{\sqrt{h}}
		\sum_{\iota\in\mathcal I}
		\gamma^\iota\sigma^\iota
		\partial_x\phi^\iota(0)
		+
		\frac12
		\sum_{\iota\in\mathcal I}
		\gamma^\iota(\sigma^\iota)^2
		\partial_{xx}\phi^\iota(0)
		+O(\sqrt{h}).
	\end{split}
	\end{equation}
	Although the singular term in
	$\mathcal G_h\phi(O)$ vanishes by the Kirchhoff condition, the remaining
	term does not, in general, converge pointwise to $\mathcal G\phi(O)$.
	The identification argument must therefore use the vanishing of
	the occupation time at the vertex, showing that the discrepancy
	between $\mathcal G_h\phi(O)$ and $\mathcal G\phi(O)$ gives no contribution in the
	limiting martingale problem.
\end{itemize}

The tightness issue discussed above is addressed in the following lemma.
\begin{lem}
	\label{lem:tightness-nonsticky}
	Assume $\eta=0$ and let $\widehat X^h$ be the piecewise-constant
	interpolation of the Markov chain defined by
	\eqref{eq:nonsticky-vertex-transition}. If $X_0^h\to x\in\Gamma$,
	then, for every $T>0$, the family $\{\widehat X^h\}_{h>0}$ 
	is tight in $D([0,T];\Gamma)$.
\end{lem}

\begin{proof}
	The compact containment condition follows from
	Lemma~\ref{lem:lyapunov}, whose proof remains unchanged when
	$\eta=0$. We therefore only need to verify Aldous' modulus
	condition.	Let
	\[
	\overline{\sigma}
	:=
	\sum_{\iota\in\mathcal I}\gamma^\iota\sigma^\iota
	\]
	and set $\Delta_n:=x_{n+1}-x_n.$	If $x_n>0$, the local consistency identities give
	\[
	\mathbb E[\Delta_n\mid\mathcal F_n]
	=
	b^{\iota_n}(x_n)h.
	\]
	If $x_n=0$, the chain is immediately re-emitted from the vertex and
	hence
	\[
	\Delta_n=\sigma^{I_n}\sqrt h,
	\qquad
	\mathbb E[\Delta_n\mid\mathcal F_n]
	=
	\overline{\sigma}\sqrt h.
	\]
	Consequently,
	\begin{equation}
		\mathbb E[\Delta_n\mid\mathcal F_n]
		=
		b^{\iota_n}(x_n)h\,\mathbf 1_{\{x_n>0\}}
		+
		\overline{\sigma}\sqrt h\,
		\mathbf 1_{\{x_n=0\}}.		
	\end{equation}
	Define
	\[
	a_n
	:=
	b^{\iota_n}(x_n)h\,\mathbf 1_{\{x_n>0\}},\qquad \xi_{n+1}
	:=
	\Delta_n-a_n
	-\overline{\sigma}\sqrt h\,
	\mathbf 1_{\{x_n=0\}}.
	\]
	Then $\mathbb E[\xi_{n+1}\mid\mathcal F_n]=0$ and
	\begin{equation}
		\mathbb E[\xi_{n+1}^2\mid\mathcal F_n]
		\leq Ch,
		\label{eq:nonsticky-martingale-variance}
	\end{equation}
	with a constant independent of $n$ and $h$. Indeed, for $x_n>0$
	this is the same variance estimate used in the proof of
	Lemma~\ref{lem:tightness}, see \eqref{eq:stima_2_moment}, whereas, for $x_n=0$,
	\[
	\begin{split}
		\mathbb E[\xi_{n+1}^2\mid\mathcal F_n]
		&=
		h\sum_{\iota\in\mathcal I}
		\gamma^\iota
		\bigl(\sigma^\iota-\overline{\sigma}\bigr)^2
		\leq Ch.
	\end{split}
	\]
	Introduce the processes
	\[
	A_n^h:=\sum_{k=0}^{n-1}a_k,
	\qquad
	M_n^h:=\sum_{k=0}^{n-1}\xi_{k+1},\qquad K_n^h
	:=
	\overline{\sigma}\sqrt h
	\sum_{k=0}^{n-1}\mathbf 1_{\{x_k=0\}}.
	\]
	Then \(M^h\) is a martingale, \(K^h\) is nondecreasing, and
	\begin{equation}
		x_n=x_0+A_n^h+M_n^h+K_n^h.
		\label{eq:nonsticky-semimartingale-decomposition}
	\end{equation}
	Set $Y_n^h:=x_0+A_n^h+M_n^h$, so that $x_n=Y_n^h+K_n^h$. 	We first establish a discrete reflection estimate. For integers
	$n\geq0$ and $m\geq0$, set
	\[
	\Omega_{n,m}^h
	:=
	\max_{0\leq j\leq m}
	|Y_{n+j}^h-Y_n^h|.
	\]
	Then
	\begin{equation}
		K_{n+m}^h-K_n^h
		\leq
		2\Omega_{n,m}^h
		+\overline{\sigma}\sqrt h.
		\label{eq:nonsticky-discrete-reflection}
	\end{equation}
	Indeed, if the chain does not visit the vertex at any index
	$k\in\{n,\ldots,n+m-1\}$, then the left-hand side vanishes. Otherwise,
	let $r$ and $\ell$ be respectively the first and the last indices in
	this set for which $x_k=0$. Since $K^h$ increases only at such
	indices,
	\[
	K_{n+m}^h-K_n^h
	=
	K_\ell^h-K_r^h+\overline{\sigma}\sqrt h.
	\]
	At every vertex visit one has $x_k=0$ and hence, by
	\eqref{eq:nonsticky-semimartingale-decomposition}, $Y_k^h=-K_k^h$.
	It follows that
	\[
	\begin{split}
		K_\ell^h-K_r^h
		=
		|Y_\ell^h-Y_r^h|
	\leq
		|Y_\ell^h-Y_n^h|
		+
		|Y_r^h-Y_n^h|
		\leq 2\Omega_{n,m}^h,
	\end{split}
	\]
	which proves \eqref{eq:nonsticky-discrete-reflection}.
	
	Let now $(\tau_h)_h$ be a family of stopping times bounded by $T$,
	and let $\theta_h\downarrow0$. Set
	\[
	\nu_h:=\lfloor\tau_h/h\rfloor,
	\qquad
	N_h:=\lfloor\theta_h/h\rfloor+1.
	\]
	Since $\widehat X^h$ is piecewise constant, there exists a random
	integer $m_h\in\{0,\ldots,N_h\}$ such that
	\[
	\widehat X^h(\tau_h+\theta_h)=X_{\nu_h+m_h}^h,
	\qquad
	\widehat X^h(\tau_h)=X_{\nu_h}^h.
	\]
	For the finite-variation part, we have
	\begin{equation}
		\max_{0\leq m\leq N_h}
		|A_{\nu_h+m}^h-A_{\nu_h}^h|
		\leq
		\|b\|_\infty N_hh.
		\label{eq:nonsticky-drift-bound}
	\end{equation}
	For the martingale part, the shifted process
	$ \bigl(
	M_{\nu_h+m}^h-M_{\nu_h}^h
	\bigr)_{m\geq0}$
	is a martingale. Therefore, Doob's inequality and
	\eqref{eq:nonsticky-martingale-variance} give
	\begin{equation}
		\begin{split}
			&\mathbb E\left[
			\max_{0\leq m\leq N_h}
			|M_{\nu_h+m}^h-M_{\nu_h}^h|^2
			\right]
		\leq
			4\,
			\mathbb E\left[
			\sum_{k=0}^{N_h-1}
			\mathbb E\left[
			\xi_{\nu_h+k+1}^2
			\mid\mathcal F_{\nu_h+k}
			\right]
			\right]
			\leq CN_hh.
			\label{eq:nonsticky-martingale-bound}
		\end{split}
	\end{equation}
	It follows from
	\eqref{eq:nonsticky-drift-bound} and
	\eqref{eq:nonsticky-martingale-bound} that
	\begin{equation}
		\begin{split}
			\mathbb E\left[
			\max_{0\leq m\leq N_h}
			|Y_{\nu_h+m}^h-Y_{\nu_h}^h|^2
			\right]
			\leq
			C\bigl((N_hh)^2+N_hh\bigr).
			\label{eq:nonsticky-Y-bound}
		\end{split}
	\end{equation}
	By the decomposition $x^h=Y^h+K^h$ and the discrete reflection
	estimate \eqref{eq:nonsticky-discrete-reflection},
	\[
	\begin{split}
		\max_{0\leq m\leq N_h}
		|x_{\nu_h+m}-x_{\nu_h}|
		&\leq
		\max_{0\leq m\leq N_h}
		|Y_{\nu_h+m}^h-Y_{\nu_h}^h|
		+
		K_{\nu_h+N_h}^h-K_{\nu_h}^h
		\\
		&\leq
		3
		\max_{0\leq m\leq N_h}
		|Y_{\nu_h+m}^h-Y_{\nu_h}^h|
		+
		\overline{\sigma}\sqrt h.
	\end{split}
	\]
	Thus, by \eqref{eq:nonsticky-Y-bound},
	\begin{equation}
		\begin{split}
			&\mathbb E\left[
			\max_{0\leq m\leq N_h}
			|x_{\nu_h+m}-x_{\nu_h}|^2
			\right]
			\leq
			C\bigl((N_hh)^2+N_hh+h \bigr).
			\label{eq:nonsticky-radial-oscillation}
		\end{split}
	\end{equation}
Since $N_hh\leq\theta_h+2h$, 	the right-hand side of
	\eqref{eq:nonsticky-radial-oscillation} converges to zero.
	
	Finally, the argument relating the network distance to the radial
	oscillation is the same as in the proof of
	Lemma~\ref{lem:tightness}. Namely, if the chain does not change edge,
	the network distance equals the radial distance. If it changes edge,
	it must visit the vertex at an intermediate step, and therefore
	\[
	d\bigl(
	X_{\nu_h+m}^h,
	X_{\nu_h}^h
	\bigr)
	\leq
	3
	\max_{0\leq j\leq m}
	|x_{\nu_h+j}-x_{\nu_h}|.
	\]
	Consequently,
	\[
	\begin{split}
		&\mathbb E\left[
		d\bigl(
		\widehat X^h(\tau_h+\theta_h),
		\widehat X^h(\tau_h)
		\bigr)^2
		\right]
		\leq
		C\bigl(
		(\theta_h+h)^2+\theta_h+h
		\bigr)
		\longrightarrow0.
	\end{split}
	\]
	Hence
	\[
	d\bigl(
	\widehat X^h(\tau_h+\theta_h),
	\widehat X^h(\tau_h)
	\bigr)
	\longrightarrow0
	\]
	in probability.	Moreover, as in the sticky case, every jump of $\widehat X^h$ is
	bounded by $\|\sigma\|_\infty\sqrt h$, and therefore the maximal jump
	size converges uniformly to zero. Together with compact containment,
	this proves Aldous' criterion \cite[Theorem 8.6-(c) p.138]{ek} and hence the tightness of
	$\{\widehat X^h\}_{h>0}$ in $D([0,T];\Gamma)$.
\end{proof}
We now identify the limit by proving an integrated consistency
property, using the fact that the discrete occupation time of the
vertex vanishes as $h\to0$.

\begin{thm}
	
Assume $\eta=0$. Let $\widehat X^h$ be the piecewise constant
interpolation \eqref{eq:interpolation} of the Markov chain
\eqref{eq:space_fpk}--\eqref{eq:index_fpk}, and assume that
$X_0^h\to x\in\Gamma$ as $h\to0$. Then, for every $T>0$,
\[
\widehat X^h \Rightarrow X
\qquad\text{in }D([0,T];\Gamma),
\]
where $X$ is the non-sticky diffusion on $\Gamma$ with generator
$(\mathcal G,D(\mathcal G))$.
\end{thm}

\begin{proof}
	Since the maximal jump size of $\widehat X^h$ is bounded by
	$\|\sigma\|_\infty\sqrt h$, every weak limit has continuous paths.

	Let $\phi\in D(\mathcal G)$, with
	$\phi^\iota\in C_b^3([0,+\infty))$ for every
	$\iota\in\mathcal I$. For the discrete chain,
	\[
	M_n^{h,\phi}
	:=
	\phi(X_n^h)-\phi(X_0^h)
	-
	h\sum_{k=0}^{n-1}\mathcal G_h\phi(X_k^h)
	\]
	is a martingale. We show that its compensator can be replaced, up to
	an error vanishing in probability, by the continuous compensator
	associated with $\mathcal G$.
	
	We first recall that the discrete occupation time of the vertex is
	\[
	\mathcal O_T^h
	:=
	h\sum_{k=0}^{\lfloor T/h\rfloor-1}
	\mathbf 1_{\{X_k^h=O\}}.
	\]
	With the notation introduced in the proof of
	Lemma~\ref{lem:tightness-nonsticky},
	\[
	\mathcal O_T^h
	=
	\frac{\sqrt h}{\overline{\sigma}}
	K_{\lfloor T/h\rfloor}^h,
	\qquad
	\overline{\sigma}
	:=
	\sum_{\iota\in\mathcal I}
	\gamma^\iota\sigma^\iota.
	\]
	The discrete reflection estimate and the drift-martingale bounds
	proved there imply
	\[
	\sup_{0<h<h_0}
	\mathbb E\left[
	\left|K_{\lfloor T/h\rfloor}^h\right|^2
	\right]
	\leq C_T.
	\]
	Consequently,
	\begin{equation}
		\mathbb E\left[
		|\mathcal O_T^h|^2
		\right]
		\leq C_T h,
		\label{eq:vanishing-discrete-occupation}
	\end{equation}
	and therefore $\mathcal O_T^h\longrightarrow0$ in $L^2$.
		
	We now estimate the error between the discrete and continuous
	generators. At the interior grid points, the consistency estimate
	already proved gives
	\begin{equation}
		\sup_{\substack{y\in\Gamma_h\\y\neq O}}
		|\mathcal G_h\phi(y)-\mathcal G\phi(y)|
		\leq C\sqrt h.
		\label{eq:nonsticky-interior-consistency}
	\end{equation}
	At the vertex, by \eqref{eq:nonsticky-discrete-generator}, although $\mathcal G_h\phi(O)$ need not converge to $\mathcal G\phi(O)$, its
	difference from $\mathcal G\phi(O)$ remains uniformly bounded:
	\begin{equation}
		|\mathcal G_h\phi(O)-\mathcal G\phi(O)|\leq C.
		\label{eq:bounded-vertex-error}
	\end{equation}
	Combining \eqref{eq:nonsticky-interior-consistency} and
	\eqref{eq:bounded-vertex-error}, for every $t\in[0,T]$ we obtain
	\begin{equation}
		h\sum_{k=0}^{\lfloor t/h\rfloor-1}
		\left|
		\mathcal G_h\phi(X_k^h)-\mathcal G\phi(X_k^h)
		\right|
		 \leq
		CT\sqrt h
		+
		C h\sum_{k=0}^{\lfloor T/h\rfloor-1}
		\mathbf 1_{\{X_k^h=O\}}=
		CT\sqrt h+C\mathcal O_T^h.	
	\end{equation}
	By \eqref{eq:vanishing-discrete-occupation}, the right-hand side
	converges to zero in probability. Thus,
	\begin{equation}
		\sup_{0\leq t\leq T}
		\left|
		h\sum_{k=0}^{\lfloor t/h\rfloor-1}
		\mathcal G_h\phi(X_k^h)
		-
		\int_0^t\mathcal G\phi(\widehat X^h(s))\,ds
		\right|
		\longrightarrow0
		\label{eq:integrated-consistency-nonsticky}
	\end{equation}
	in probability, where the additional error due to the last
	incomplete time interval is bounded by $h\|\mathcal G\phi\|_\infty$. From this point onwards, the proof proceeds as in Theorem \ref{thm:EK_convergence}.

\end{proof}

\section{Semi-Lagrangian Scheme for the HJB equation}\label{sec:HJB_equation}
We now introduce the stochastic control problem whose dynamic programming
equation will be approximated below. Let $A$ be a compact metric space,
let $\lambda>0$ be the discount factor, and let $\theta\in\mathbb R$
denote the running cost incurred at the vertex. On each edge
$\Gamma^\iota$, the drift and the running cost are now controlled
coefficients
\[
b^\iota=b^\iota(x,a),
\qquad
f^\iota=f^\iota(x,a),
\qquad
(x,a)\in[0,+\infty)\times A.
\]
We assume that $b^\iota$ and $f^\iota$ are continuous and bounded, and
 Lipschitz continuous in $x$, uniformly with
respect to $a\in A$. The diffusion coefficients $\sigma^\iota$, the
stickiness parameter $\eta$, and the redistribution weights
$\gamma^\iota$ are not controlled and satisfy the assumptions introduced
in \Cref{sec:prelim}.

Let
\[
PC(\Gamma;A)
:=
\left\{
a:\Gamma\to A:
a^\iota\in C((0,+\infty);A)
\text{ and }
\lim_{x\downarrow0}a^\iota(x)
\text{ exists in }A
\text{ for every }\iota\in\mathcal I
\right\}.
\]
We denote by
\[
\mathcal A:=PC(\Gamma;A)
\]
the set of stationary feedback controls.
For $a\in\mathcal A$, let
$X^a(t)=(\iota(t),x(t))$ be the controlled sticky diffusion generated by $(\mathcal{G}_a, D(\mathcal{G}_a))$, where
\[
D(\mathcal G_a)=
\left\{
\varphi \in C^2(\Gamma)\cap C_0(\Gamma):
\mathcal G_a \varphi \in C_0(\Gamma),\quad
\eta \mathcal G_a\varphi (O)
=
\sum_{\iota\in\mathcal I}
\gamma^\iota\sigma^\iota \partial_x \varphi^\iota(0)
\right\}
\]
where, for $\iota\in\mathcal{I}$, the generator $\mathcal G_a$ is given by
\[
\mathcal G^\iota_a f(x)=\frac12(\sigma^\iota)^2 \partial_{xx} f^\iota(x)+b^\iota(x,a^\iota(x))\partial_x f^\iota(x),
\qquad x>0.
\]
We recall from \cite{fs,BC1} that there exists a continuous and nondecreasing process $L^a$, with $L^a(0)=0$ and increasing
only when $X^a(t)=O$, such that the pair $(X^a,L^a)$ is a weak solution to
\begin{equation}
	\label{eq:controlled-sticky-diffusion}
	\left\{
	\begin{aligned}
	&	dx(t)
		=
		b^{\iota(t)}
		\bigl(x(t),a(X^a(t))\bigr)
		\mathbf 1_{\{X^a(t)\neq O\}}\,dt
		+
		\sigma^{\iota(t)}
		\mathbf 1_{\{X^a(t)\neq O\}}\,dW_t
		+\biggl(\sum_{\kappa\in\mathcal I}
	\gamma^\kappa\sigma^\kappa\biggr)dL^a(t),
		\\[4pt]
	&	\int_0^t
		\mathbf 1_{\{X^a(s)=O\}}\,ds
		=
		\eta L^a(t).
	\end{aligned}
	\right.
\end{equation}

The edge-label process obeys the redistribution
rule introduced in \Cref{sec:prelim}: whenever a new excursion starts
from the vertex,
\[
\mathbb P\bigl(
\iota(\tau^+)=\iota
\mid X^a(\tau)=O
\bigr)
=
\frac{\gamma^\iota\sigma^\iota}
{\sum_{\kappa\in\mathcal I}
	\gamma^\kappa\sigma^\kappa}.
\]
In particular, the control acts only on the drift along the open edges
and does not affect either the sticky mechanism or the redistribution
at the vertex.

For an initial state $x\in\Gamma$ and a feedback control
$a\in\mathcal A$, define the discounted cost
\[
J(x,a)
:=
\mathbb E_x
\left[
\int_0^{+\infty}
e^{-\lambda s}
\left(
f^{\iota(s)}
\bigl(x(s),a(X^a(s))\bigr)
\mathbf 1_{\{X^a(s)\neq O\}}
+
\theta\mathbf 1_{\{X^a(s)=O\}}
\right)ds
\right],
\]
and the value function
\[
u(x):=\inf_{a\in\mathcal A}J(x,a).
\]

From \cite{adlt,BC2}, we know that the corresponding Hamiltonian and Hamilton-Jacobi-Bellman equation are
\[
H^\iota(x,p)
:=
\sup_{a\in A}
\left\{
-b^\iota(x,a)p-f^\iota(x,a)
\right\}.
\]
and
\begin{equation}
	\label{eq:HJB_control}
	\begin{cases}
		\displaystyle
		-\frac12(\sigma^\iota)^2\partial_{xx}u^\iota
		+H^\iota(x,\partial_xu^\iota)
		+\lambda u^\iota=0,
		& x>0,\quad \iota\in\mathcal I,
		\\[5pt]
		u^\iota(0)=u^\kappa(0)=:u(O),
		& \iota,\kappa\in\mathcal I,
		\\[5pt]
		\displaystyle
		\sum_{\iota\in\mathcal I}
		\gamma^\iota\sigma^\iota\partial_xu^\iota(0)
		=
		\eta\bigl(\lambda u(O) - \theta \bigr).
	\end{cases}
\end{equation}

For the viscosity formulation, define
\begin{align}
F^\iota(x,r,p,X)
&:=
-\frac12(\sigma^\iota)^2 X+H^\iota(x,p)+\lambda r,
\qquad x\ge0,
\\
K(r,p)
&:=
-\sum_{\iota\in\mathcal I}\gamma^\iota\sigma^\iota p^\iota
+\eta(\lambda r-\theta),
\qquad p=(p^\iota)_{\iota\in\mathcal I}.
\end{align}
For $x\in\Gamma$, $r\in\mathbb R$, and
$p=(p^\iota)_{\iota\in\mathcal I}$,
$X=(X^\iota)_{\iota\in\mathcal I}$, set
\[
F^*(x,r,p,X)
:=
\begin{cases}
F^\iota(s,r,p^\iota,X^\iota),
& x=(\iota,s),\ s>0,\\[3pt]
\displaystyle
\max\left\{
K(r,p),\max_{\iota\in\mathcal I}F^\iota(0,r,p^\iota,X^\iota)
\right\},
& x=O,
\end{cases}
\]
and
\[
F_*(x,r,p,X)
:=
\begin{cases}
F^\iota(s,r,p^\iota,X^\iota),
& x=(\iota,s),\ s>0,\\[3pt]
\displaystyle
\min\left\{
K(r,p),\min_{\iota\in\mathcal I}F^\iota(0,r,p^\iota,X^\iota)
\right\},
& x=O.
\end{cases}
\]
At the vertex all derivatives are understood as one-sided derivatives along
the corresponding edge.

\begin{defi}\label{sub-super-sol}
A function $u\in USC(\Gamma)$ is a viscosity subsolution of
\eqref{eq:HJB_control} if, for every test function $\varphi\in C^2(\Gamma)$
and every local maximum point $x_0$ of $u-\varphi$, the following conditions
hold.

If $x_0=(\iota,s_0)$ with $s_0>0$, then
\[
F^\iota\bigl(s_0,u(x_0),\partial_x\varphi^\iota(s_0),
\partial_{xx}\varphi^\iota(s_0)\bigr)\le0.
\]
If $x_0=O$, then
\[
F_*\Bigl(O,u(O),
(\partial_x\varphi^\iota(0))_{\iota\in\mathcal I},
(\partial_{xx}\varphi^\iota(0))_{\iota\in\mathcal I}\Bigr)\le0.
\]

A function $u\in LSC(\Gamma)$ is a viscosity supersolution of
\eqref{eq:HJB_control} if, for every $\varphi\in C^2(\Gamma)$ and every local
minimum point $x_0$ of $u-\varphi$, the reverse inequalities hold, with
$F_*$ replaced by $F^*$ at the vertex. A function $u\in C(\Gamma)$ is a
viscosity solution if it is both a subsolution and a supersolution.
\end{defi}
The following comparison principle follows by a straightforward adaptation to the unbounded case of the corresponding result in \cite{BLT}.
\begin{thm}\label{thm:comparison}
Let $u\in USC(\Gamma)$ be a bounded viscosity subsolution of
\eqref{eq:HJB_control}, and let $v\in LSC(\Gamma)$ be a bounded viscosity
supersolution of \eqref{eq:HJB_control}. Then $u\le v$ on $\Gamma$. Hence the
bounded viscosity solution is unique.
\end{thm}

\subsection{The edge-adapted approximation scheme}
We approximate \eqref{eq:HJB_control} by adapting the edge-based Markov
chain introduced in \Cref{sec:euler-Mar} to the controlled drift
$b^\iota(x,a)$. The edge-adapted grid and the vertex transition
mechanism remain unchanged.  All bounds and restrictions on the discretization parameter
are understood uniformly with respect to $a\in A$.
For $h\in(0,h_0)$ define
\begin{equation}
p_\pm^\iota(x,a)
:=
\frac12\left(1\pm\frac{\sqrt h}{\sigma^\iota}b^\iota(x,a)\right),
\qquad x\ge0,\ a\in A.
\end{equation}
By the choice of $h_0$, $p_\pm^\iota(x,a)\in[0,1]$ and
$p_+^\iota(x,a)+p_-^\iota(x,a)=1$. The grid is
\begin{equation}
\Gamma_h
:=
\{O\}\cup
\bigcup_{\iota\in\mathcal I}
\bigl\{(\iota,j\sigma^\iota\sqrt h):\ j\in\mathbb N,\ j\ge1\bigr\}.
\end{equation}
Let $\mathcal X_h$ be the space of bounded grid functions
$w=(w_0,(w_j^\iota)_{\iota\in\mathcal I,\ j\ge1})$, where $w_0$ is the
common vertex value. We use the convention $w_0^\iota=w_0$. On each edge,
$I_h[w]$ denotes the continuous, piecewise affine interpolant defined by
\[
I_h[w]^\iota(x)
=
\frac{x_{j+1}^\iota-x}{\sigma^\iota\sqrt h}w_j^\iota
+
\frac{x-x_j^\iota}{\sigma^\iota\sqrt h}w_{j+1}^\iota,
\qquad
x\in[x_j^\iota,x_{j+1}^\iota],
\]
where $x_j^\iota:=j\sigma^\iota\sqrt h$ and $j\ge0$.

The controlled transition operator $\mathcal P_h(a):\mathcal X_h\to\mathcal X_h$
is given, for $j\ge1$, by
\begin{equation}
[\mathcal P_h(a)w]_j^\iota
:=
p_+^\iota(x_j^\iota,a)w_{j+1}^\iota
+
p_-^\iota(x_j^\iota,a)w_{j-1}^\iota,
\end{equation}
and at the vertex by
\begin{equation}
[\mathcal P_hw]_0
:=
\frac{\eta}{\eta+\sqrt h}w_0
+
\frac{\sqrt h}{\eta+\sqrt h}
\sum_{\kappa\in\mathcal I}\gamma^\kappa w_1^\kappa.
\end{equation}
The vertex transition is not controlled.

Set $\rho_h:=1-\lambda h$. The approximation scheme with time step is $h$  and the spatial mesh  $\sigma^\iota\sqrt h$ on
$\Gamma^\iota$  
reads as follows: \\
Find $u_h\in\mathcal X_h$
such that, for every $\iota\in\mathcal I$ and $j\ge1$,
\begin{align}\label{eq:SL_fullydiscrete_edge}
&u_{h,j}^\iota
=
\inf_{a\in A}\left\{
hf^\iota(x_j^\iota,a)
+
\rho_h\left[
p_+^\iota(x_j^\iota,a)u_{h,j+1}^\iota
+
p_-^\iota(x_j^\iota,a)u_{h,j-1}^\iota
\right]
\right\},\\
\label{eq:SL_fullydiscrete_vertex}
&u_{h,0}
=
\frac{\eta\theta h}{\eta+\sqrt h}
+
\rho_h\left[
\frac{\eta}{\eta+\sqrt h}u_{h,0}
+
\frac{\sqrt h}{\eta+\sqrt h}
\sum_{\kappa\in\mathcal I}\gamma^\kappa u_{h,1}^\kappa
\right].
\end{align}
Define for $j\ge1$
\begin{align}
\mathcal S_h^\iota(x_j^\iota,r,w)
:=
\lambda r
+
\sup_{a\in A}\left\{
-\frac{\rho_h}{h}
\left[
p_+^\iota(x_j^\iota,a)w_{j+1}^\iota
+
p_-^\iota(x_j^\iota,a)w_{j-1}^\iota
-r
\right]
-f^\iota(x_j^\iota,a)
\right\},
\end{align}
and
\begin{equation}\label{eq:S_vertex_hjb}
\mathcal S_h^0(r,w)
:=
-\frac{\rho_h}{\sqrt h}
\sum_{\kappa\in\mathcal I}\gamma^\kappa
(w_1^\kappa-r)
+
\eta(\lambda r-\theta)
+
\lambda\sqrt h\,r.
\end{equation}
Then \eqref{eq:SL_fullydiscrete_edge}--\eqref{eq:SL_fullydiscrete_vertex}
are equivalent to
\begin{align*}
&\mathcal S_h^\iota(x_j^\iota,u_{h,j}^\iota,u_h)=0,
\qquad \iota\in\mathcal I,
\quad j\ge1,\\
&\mathcal S_h^0(u_{h,0},u_h)=0.
\end{align*}


The proofs of the following lemmas are given in
Appendix \ref{app:proofs_hjb_scheme}.

\begin{lem}\label{lem:stability_hjb}
Assume $0<h<\min\{h_0,\lambda^{-1}\}$. Then the scheme
\eqref{eq:SL_fullydiscrete_edge}--\eqref{eq:SL_fullydiscrete_vertex} has a
unique solution $u_h\in\mathcal X_h$. Moreover,
\begin{equation}\label{eq:stability_bound_hjb}
\|u_h\|_\infty
\le
\frac{\|f\|_\infty+|\theta|}{\lambda}.
\end{equation}
\end{lem}

\begin{lem}\label{lem:monotonicity_hjb}
Assume $0<h<\min\{h_0,\lambda^{-1}\}$. If $v,w\in\mathcal X_h$ and
$v\ge w$ on $\Gamma_h$, then
\[
\mathcal S_h^\iota(x_j^\iota,r,v)
\le
\mathcal S_h^\iota(x_j^\iota,r,w)
\]
for all $\iota\in\mathcal I$, $j\ge1$, and $r\in\mathbb R$, and
\[
\mathcal S_h^0(r,v)\le\mathcal S_h^0(r,w)
\]
for all $r\in\mathbb R$.
\end{lem}

For $\varphi\in C^2(\Gamma)$, denote by $\varphi_h\in\mathcal X_h$ its
restriction to the grid:
\[
(\varphi_h)_0:=\varphi(O),
\qquad
(\varphi_h)_j^\iota:=\varphi^\iota(x_j^\iota),
\quad j\ge1.
\]
If $y\in\Gamma_h$, define
\[
\mathcal S_h(y,r,w)
:=
\begin{cases}
\mathcal S_h^\iota(x_j^\iota,r,w),
& y=(\iota,x_j^\iota),\ j\ge1,\\[2pt]
\mathcal S_h^0(r,w),
& y=O.
\end{cases}
\]

\begin{lem}\label{lem:consistency_hjb}
Let $\varphi\in C^2(\Gamma)$. Then, for every $x\in\Gamma$,
\begin{equation}\label{eq:consistency_hjb_sub}
\limsup_{\substack{h\to0,\ \xi\to0\\ y\in\Gamma_h,\ y\to x}}
\mathcal S_h\bigl(y,\varphi(y)+\xi,\varphi_h+\xi\bigr)
\le
F^*\bigl(x,\varphi(x),D\varphi(x),D^2\varphi(x)\bigr),
\end{equation}
and
\begin{equation}\label{eq:consistency_hjb_super}
\liminf_{\substack{h\to0,\ \xi\to0\\ y\in\Gamma_h,\ y\to x}}
\mathcal S_h\bigl(y,\varphi(y)+\xi,\varphi_h+\xi\bigr)
\ge
F_*\bigl(x,\varphi(x),D\varphi(x),D^2\varphi(x)\bigr).
\end{equation}
At an interior point $x=(\iota,s)$, $s>0$, the symbols $D\varphi(x)$ and
$D^2\varphi(x)$ mean the corresponding edge derivatives; at $O$ they mean the
vectors of all outgoing first and second derivatives.
\end{lem}

\begin{thm}
Let $u$ be the unique bounded viscosity solution of \eqref{eq:HJB_control}.
For every $h\in(0,\min\{h_0,\lambda^{-1}\})$, let $u_h\in\mathcal X_h$ be
the unique solution of
\eqref{eq:SL_fullydiscrete_edge}--\eqref{eq:SL_fullydiscrete_vertex}. Then
\[
I_h[u_h]\to u
\qquad\text{locally uniformly on }\Gamma
\]
as $h\to0$.
\end{thm}

\begin{proof}
By \Cref{lem:stability_hjb}, the family $\{u_h\}_h$ is uniformly bounded.
Define the half-relaxed limits
\[
\overline u(x):=
\limsup_{\substack{h\to0\\ y\in\Gamma_h,\ y\to x}}u_h(y),
\qquad
\underline u(x):=
\liminf_{\substack{h\to0\\ y\in\Gamma_h,\ y\to x}}u_h(y).
\]
The standard Barles--Souganidis argument \cite{BS}, based on
\Cref{lem:monotonicity_hjb,lem:consistency_hjb}, shows that $\overline u$ is
a viscosity subsolution and $\underline u$ is a viscosity supersolution of
\eqref{eq:HJB_control}. At the vertex, the relaxed consistency inequalities
produce exactly the pair $F_*$ and $F^*$ used in
Definition~\ref{sub-super-sol}.

By \Cref{thm:comparison}, $\overline u\le\underline u$ on $\Gamma$. Since the
reverse inequality follows from the definitions, $\overline u=\underline u=u$.
Therefore $u_h\to u$ locally uniformly on the grid. The local uniform
convergence of the interpolants $I_h[u_h]$ follows because the mesh size on
$\Gamma^\iota$ is $\sigma^\iota\sqrt h\to0$ and $u$ is locally uniformly
continuous on $\Gamma$.
\end{proof}
\section{Numerical experiments}
We present a few numerical tests illustrating the approximation
results obtained in the previous sections. Our purpose is to emphasize
the main qualitative and structural properties of the Markov-chain
approximation and of the associated HJB scheme, rather than to
perform a systematic error analysis.
\subsection{Simulation of the approximating chain}
\label{subsec:numerical-chain}

We illustrate some structural properties of the Markov-chain
approximation on a star graph with three edges. We take
\[
b^\iota\equiv0,
\qquad
(\sigma^1,\sigma^2,\sigma^3)=(0.8,1,1.4),
\qquad
(\gamma^1,\gamma^2,\gamma^3)=(0.25,0.45,0.30).
\]
All simulations start from the vertex. The choice of a vanishing drift
allows us to compare the numerical results directly with the
calibration properties discussed in
Remark~\ref{rem:calibration}.

We first compare sample paths of the radial component in the
non-sticky and sticky cases. Figure~\ref{fig:chain-sample-paths}
shows two realizations obtained with $h=2^{-10}$. When $\eta=0$, each
visit to the vertex is followed by immediate re-emission. For
$\eta>0$, the geometric holding mechanism produces visible intervals
during which the chain remains at the vertex.

\begin{figure}[t]
	\centering
	\includegraphics[width=0.76\textwidth]
	{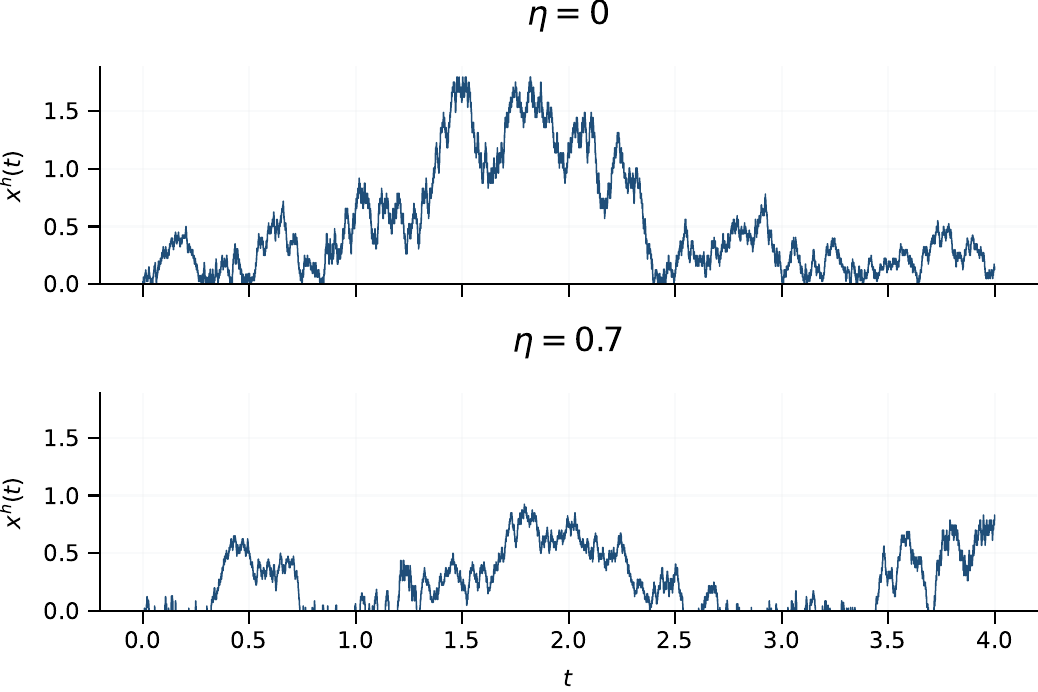}
	\caption{
		Sample paths of the radial component of the approximating
		chain for $\eta=0$ and $\eta=0.35$, with $h=2^{-10}$.
		The horizontal portions at zero correspond to discrete
		residence times at the vertex.
	}
	\label{fig:chain-sample-paths}
\end{figure}

We next consider the discrete occupation time
\[
\mathcal O_T^h
:=
h\sum_{n=0}^{\lfloor T/h\rfloor-1}
\mathbf 1_{\{X_n^h=O\}}.
\]
The left panel of Figure~\ref{fig:chain-structural-tests} reports its
empirical mean at $T=3$ for different values of $\eta$. Each value is
estimated using $3000$ independent trajectories, and the vertical
bars represent $95\%$ confidence intervals. In the non-sticky case,
the occupation time decreases to zero as the grid is refined. A
logarithmic regression gives the empirical behavior
\[
\mathbb E[\mathcal O_T^h]\simeq C h^{0.499},
\]
which is consistent with the $O(\sqrt h)$ estimate suggested by the
discrete reflection decomposition. For $\eta>0$, the occupation time
remains positive and increases with the stickiness parameter.

Finally, we test the redistribution mechanism. Starting from $O$, we
record the edge along which the chain first exits the ball
$B_\rho(O)$, with $\rho=0.8$. In the drift-free case, the limiting
exit probabilities are
\[
q^\iota
=
\frac{\gamma^\iota\sigma^\iota}
{\sum_{\kappa\in\mathcal I}
	\gamma^\kappa\sigma^\kappa},
\qquad \iota\in\mathcal I.
\]
For the parameters above,
\[
(q^1,q^2,q^3)
\simeq
(0.187,0.421,0.393).
\]
The right panel of Figure~\ref{fig:chain-structural-tests} compares
these values with the empirical frequencies obtained from $10^4$
independent exits. The dashed lines denote the theoretical
probabilities. The empirical frequencies remain close to their
predicted values for all the considered discretization parameters.

\begin{figure}[t]
	\centering
	\includegraphics[width=\textwidth]
	{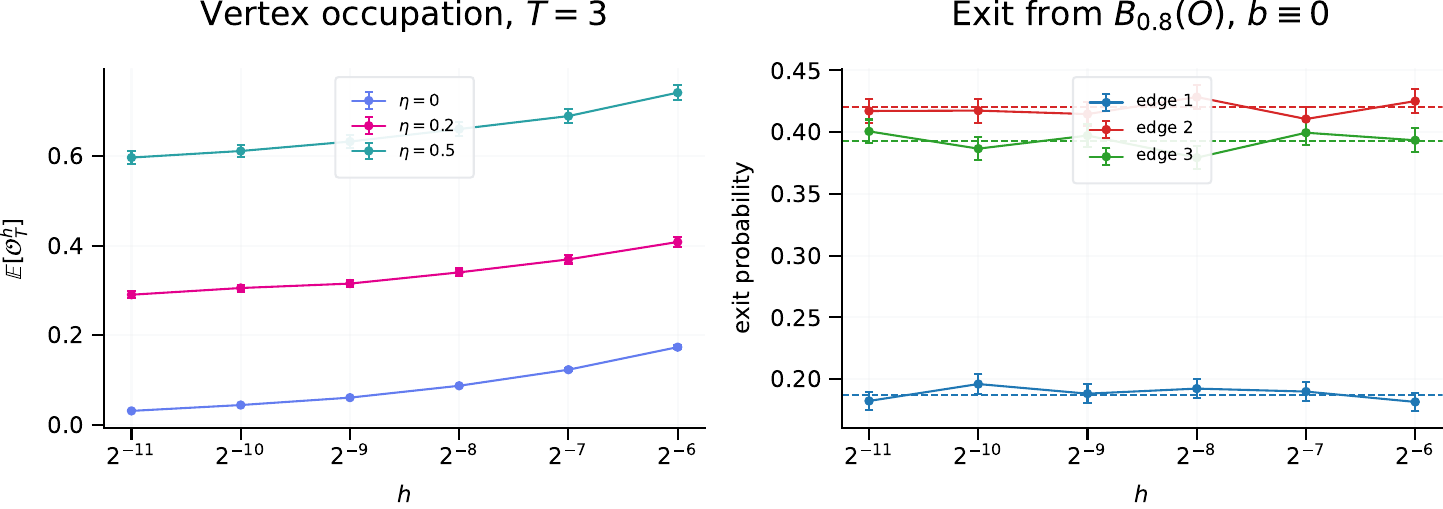}
	\caption{
		Left: empirical mean of the discrete occupation time at the
		vertex for different values of $\eta$. Right: empirical
		probabilities of exiting $B_{0.8}(O)$ through each edge; the
		dashed lines represent the limiting probabilities
		$q^\iota$.
	}
	\label{fig:chain-structural-tests}
\end{figure}

These experiments illustrate the two main effects encoded by the
vertex transition mechanism: the approximation of the sticky
occupation time and the correct redistribution of excursions among
the incident edges. They also support the vanishing occupation-time
property used in the convergence analysis of the non-sticky case.

\subsection{Numerical approximation of the value function}
\label{subsec:numerical-hjb}

We finally illustrate the behavior of the HJB scheme on the
three-edge network considered above. We take
\[
A=[-1,1],
\qquad
b^\iota(x,a)=a,
\qquad
f^\iota(x,a)
=
c^\iota e^{-x}+\frac{\nu}{2}a^2,
\]
where
\[
(c^1,c^2,c^3)=(1,0.7,1.3),
\qquad
\nu=0.2,
\qquad
\lambda=1,
\qquad
\theta=0.5,
\qquad
\eta=0.5.
\]
The control acts on the drift along the edges, while the diffusion
coefficients and the vertex transition mechanism are the same as in
the previous experiment.

The computation is performed on the truncated network
$\Gamma_R$, with $R=8$, imposing the artificial boundary condition
\[
u^\iota(R)=0,
\qquad \iota\in\mathcal I.
\]
The discrete HJB equation is solved by policy iteration. Since no
explicit solution is available, we compute a reference solution
$u_{\mathrm{ref}}$ using $h_{\mathrm{ref}}=2^{-14}$. For the coarser
discretizations, the error is evaluated on the interior sub-network
$\Gamma_R$ according to
\[
E_h
:=
\max_{\substack{\iota\in\mathcal I\\x_j^\iota\leq 6}}
\left|
u_{h,j}^\iota
-
I_{h_{\mathrm{ref}}}
[u_{\mathrm{ref}}]^\iota(x_j^\iota)
\right|.
\]

Figure~\ref{fig:numerical-value-function} shows the reference
value-function profiles on the three edges and the decay of $E_h$
under grid refinement. The different profiles reflect the
edge-dependent running costs and diffusion coefficients. The errors
decrease regularly as $h\to0$; a logarithmic fit gives an empirical
slope of approximately $0.64$. This value is reported only as a
numerical observation and is not intended as a theoretical
convergence rate.

\begin{figure}[t]
	\centering
	\includegraphics[width=\textwidth]
	{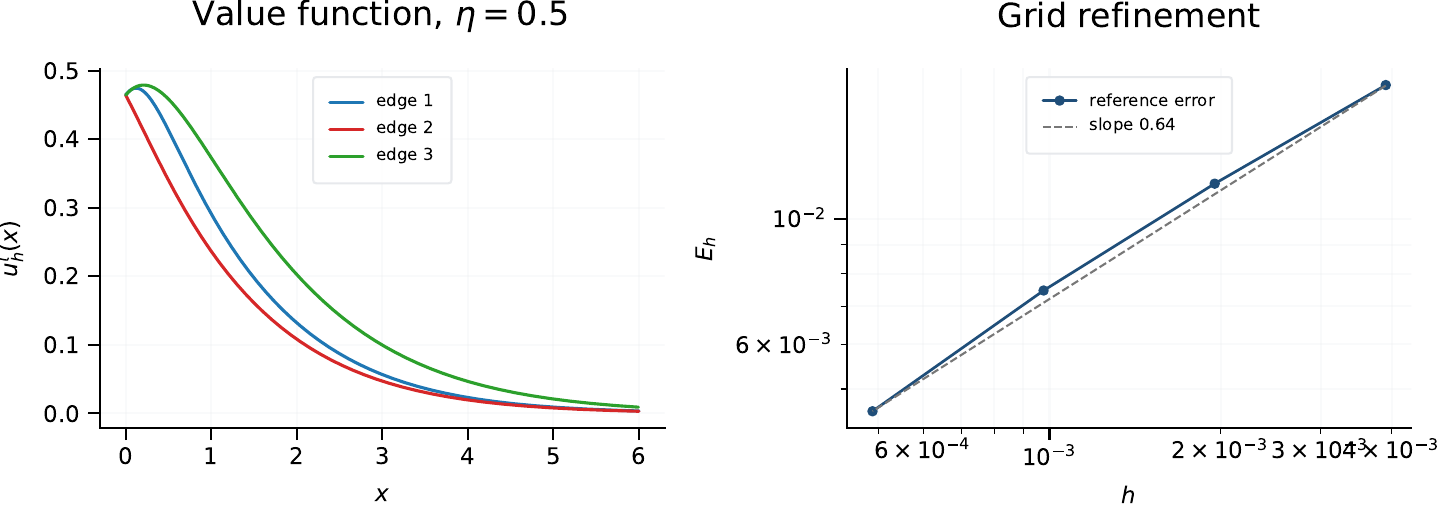}
	\caption{
		Left: numerical value function on the three edges for
		$\eta=0.5$, computed with $h_{\mathrm{ref}}=2^{-14}$.
		Right: error $E_h$ with respect to the reference solution
		on the interior sub-network $\Gamma\cap \{x^\iota_j\le 6\}$. The dashed line
		represents the empirical slope $0.64$.
	}
	\label{fig:numerical-value-function}
\end{figure}
We next investigate the influence of the stickiness parameter,
keeping all the other data unchanged. The discrete HJB equation is
solved with $h=2^{-13}$ for $\eta\in\{0,1,5,20\}$. To keep the comparison readable, the left panel of
Figure~\ref{fig:value-function-large-eta} displays the corresponding
profiles on the first edge, whereas the right panel reports the
common vertex value $u_h(O)$ for $\eta\in[0,20]$. The effect of
stickiness is most visible near the vertex and gradually decreases
along the edge.

For the selected data, the vertex value increases from approximately
$0.455$ in the non-sticky case to $0.496$ for $\eta=20$. Moreover,
the numerical values approach
\[
\frac{\theta}{\lambda}=0.5
\]
as $\eta$ increases. This behavior is consistent with the vertex
condition
\[
-\sum_{\iota\in\mathcal I}
\gamma^\iota\sigma^\iota\partial_xu^\iota(0)
=
\eta\bigl(\theta-\lambda u(O)\bigr),
\]
which formally suggests that
$u(O)\to\theta/\lambda$ in the large-stickiness regime, provided that
the outgoing derivatives remain bounded.

\begin{figure}[t]
	\centering
	\includegraphics[width=\textwidth]
	{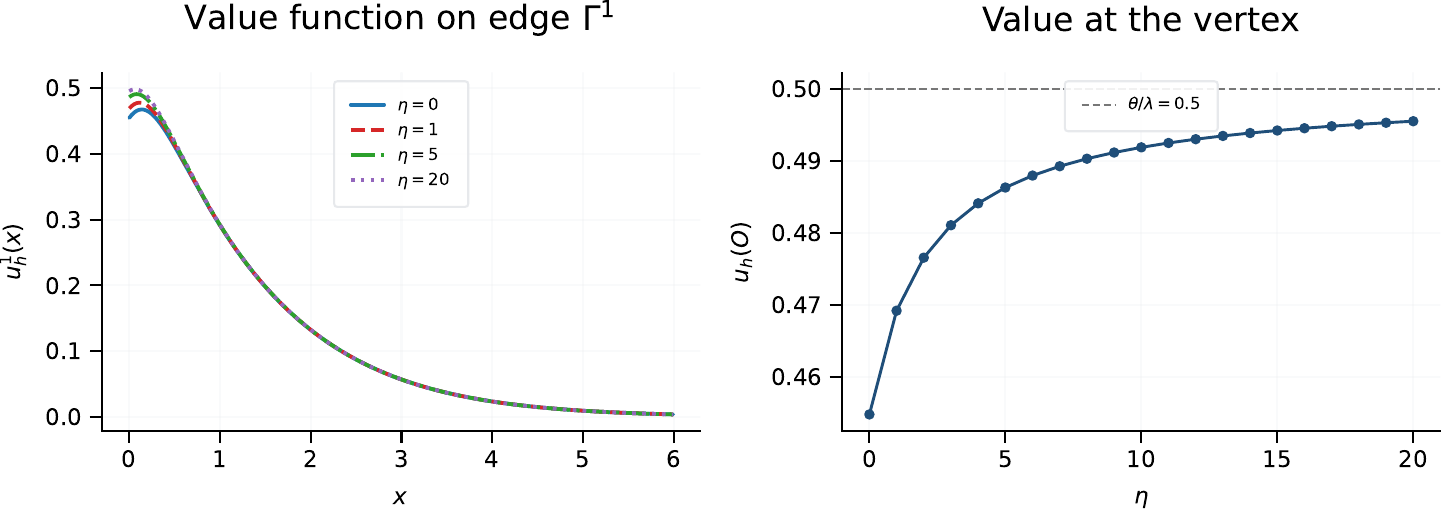}
	\caption{
		Left: numerical value function on the first edge for
		$\eta=0$, $1$, $5$, and $20$. Right: the vertex value
		$u_h(O)$ as a function of $\eta$; the dashed line represents
		$\theta/\lambda=0.5$. The computations use $h=2^{-13}$ on
		the truncated network $\Gamma_R$.
	}
	\label{fig:value-function-large-eta}
\end{figure}
\appendix
\section{Proofs for the convergence of the random-walk scheme}\label{app:proof_Euler}

\begin{proof}[Proof of \Cref{lem:lyapunov}]
We first prove \eqref{lyap_est}, distinguishing the cases $x>0$ and $x=0$.
Let $x>0$ and set $s:=\sigma^\iota\sqrt h$. Then, with
$p_\pm:=p^\iota_\pm(x)$ as in \eqref{eq:upwind_prob},
\[
\mathbb E_{\iota,x}[V(X_1^h)]-V(\iota,x)
=p_+\bigl[(x+s)^2-x^2\bigr]+p_-\bigl[(x-s)^2-x^2\bigr]
=(p_+-p_-)\,2xs+s^2
=2x\,b^\iota(x)\,h+(\sigma^\iota)^2h .
\]
Hence
\[
\mathcal G_hV(\iota,x)
=2x\,b^\iota(x)+(\sigma^\iota)^2
\le C(1+x^2)=CV(\iota,x),
\]
by the boundedness of $b$ and $\sigma$.
Let now $x=0$. Then
\[
\mathbb E_{O}[V(X_1^h)]-V(O)
=\frac{\sqrt h}{\eta+\sqrt h}
\sum_{\kappa\in\mathcal I}\gamma^\kappa
\bigl[V(\kappa,\sigma^\kappa\sqrt h)-V(O)\bigr]
=\frac{\sqrt h}{\eta+\sqrt h}
\sum_{\kappa\in\mathcal I}\gamma^\kappa(\sigma^\kappa)^2h
\le \|\sigma\|_\infty^2\,h ,
\]
so that $\mathcal G_hV(O)\le \|\sigma\|_\infty^2\le CV(O)$. Combining the
two estimates gives \eqref{lyap_est}, possibly increasing $C$.
The proof of \eqref{compact} from \eqref{lyap_est} is unchanged with
respect to the supermartingale argument: from
$\mathbb E[V(X_{n+1}^h)\mid X_n^h]\le(1+Ch)V(X_n^h)$, the process
$Y_n^h:=(1+Ch)^{-n}V(X_n^h)$ is a nonnegative supermartingale; setting
$K_R:=\{O\}\cup\{(\iota,x):\iota\in\mathcal I,\ 0\le x\le R\}$, Doob's
maximal inequality gives
\[
\mathbb P\left(
\max_{0\le n\le \lfloor T/h\rfloor}V(X_n^h)>1+R^2
\right)
\le
\frac{(1+Ch)^{\lfloor T/h\rfloor}}{1+R^2}\,V(X_0^h)
\le \frac{e^{CT}}{1+R^2}\,\sup_{0<h<h_0}V(X^h_0),
\]
where the last supremum is finite since $X^h_0\to x$. Choosing $R$ large
enough yields \eqref{compact}.
\end{proof}
\begin{proof}[Proof of \Cref{lem:tightness}]
By \Cref{lem:lyapunov}, the compact containment condition \eqref{compact}
holds. Let $(\tau_h)_h$ be stopping times bounded by $T$ and let
$\theta_h\to0$. We prove that
\[
d\bigl(\widehat X^h(\tau_h+\theta_h),\widehat X^h(\tau_h)\bigr)
\to 0
\quad\text{in probability}.
\]

Set $\Delta_n:=x_{n+1}-x_n$ and let $(\mathcal F_n)_{n\ge0}$ be the natural
filtration of the chain. Define
\[
a_n:=\mathbb E[\Delta_n\mid\mathcal F_n],
\qquad
\xi_{n+1}:=\Delta_n-a_n .
\]
By \eqref{eq:local_consistency}, $|a_n|\le \|b\|_\infty h$ on $\{x_n>0\}$,
while on $\{x_n=0\}$,
\[
|a_n|
=
\mathbb P(R_n=1)\sum_{\iota}\gamma^\iota\sigma^\iota\sqrt h
=
\frac{h}{\eta+\sqrt h}\sum_{\iota}\gamma^\iota\sigma^\iota
\le \frac{h}{\eta}\sum_{\iota}\gamma^\iota\sigma^\iota,
\]
so that, in all cases,
\[
|a_n|\le Ch
\qquad\text{for some $C > 0$ independent of $h$. }
\]
Moreover, $(\xi_n)_{n\ge1}$ are martingale
differences and
\begin{equation}\label{eq:stima_2_moment}
	\mathbb E[\xi_{n+1}^2\mid\mathcal F_n]
=
\operatorname{Var}(\Delta_n\mid\mathcal F_n)
\le \mathbb E[\Delta_n^2\mid\mathcal F_n]
\le \|\sigma\|_\infty^2\,h .
\end{equation}
Let $\nu_h:=\lfloor\tau_h/h\rfloor$ and $N_h:=\lfloor\theta_h/h\rfloor+1$. Since $\widehat X^h$ is piecewise constant and
\[
\nu_h\;\le\;\Bigl\lfloor\frac{\tau_h+\theta_h}{h}\Bigr\rfloor\;\le\;\nu_h+N_h,
\]
there exists a (random) integer $m_h\in\{0,\dots,N_h\}$ such that
\[
\widehat X^h(\tau_h+\theta_h)=X^h_{\nu_h+m_h},
\qquad
\widehat X^h(\tau_h)=X^h_{\nu_h}.
\]
Set
\[
\omega_h
:=
\max_{0\le m\le N_h}
\bigl|x_{\nu_h+m}-x_{\nu_h}\bigr|.
\]
We first relate the network distance to $\omega_h$. If the chain does not change edge between $\nu_h$ and $\nu_h+m_h$,
then
\[
d\bigl(X^h_{\nu_h+m_h},X^h_{\nu_h}\bigr)
=
|x_{\nu_h+m_h}-x_{\nu_h}|
\le \omega_h.
\]
If the chain does change edge, then the transition mechanism implies that it
must visit the vertex at some intermediate step. Hence for some
$m^\ast\in\{0,\dots,m_h\}$ one has $x_{\nu_h+m^\ast}=0$, so that
\[
\omega_h\ge |x_{\nu_h+m^\ast}-x_{\nu_h}|=x_{\nu_h}.
\]
Since also
\[
x_{\nu_h+m_h}\le x_{\nu_h}+\omega_h\le 2\omega_h,
\]
we obtain
\[
d\bigl(X^h_{\nu_h+m_h},X^h_{\nu_h}\bigr)
=
x_{\nu_h}+x_{\nu_h+m_h}
\le 3\omega_h.
\]
In both cases, therefore,
\[
d\bigl(\widehat X^h(\tau_h+\theta_h),\widehat X^h(\tau_h)\bigr)
\le 3\omega_h.
\]

Now write, for $m=0,\dots,N_h$,
\[
x_{\nu_h+m}-x_{\nu_h}
=
\sum_{k=0}^{m-1}\Delta_{\nu_h+k}
=
\sum_{k=0}^{m-1}a_{\nu_h+k}
+
\sum_{k=0}^{m-1}\xi_{\nu_h+k+1}.
\]
Hence
\[
\omega_h
\le
\sum_{k=0}^{N_h-1}|a_{\nu_h+k}|
+
\max_{0\le m\le N_h}
\left|
\sum_{k=0}^{m-1}\xi_{\nu_h+k+1}
\right|.
\]
Using $(u+v)^2\le 2u^2+2v^2$,
\[
\mathbb E[\omega_h^2]
\le
2\mathbb E\Bigl[\Bigl(\sum_{k=0}^{N_h-1}|a_{\nu_h+k}|\Bigr)^2\Bigr]
+
2\mathbb E\Bigl[
\max_{0\le m\le N_h}
\Bigl|
\sum_{k=0}^{m-1}\xi_{\nu_h+k+1}
\Bigr|^2
\Bigr].
\]
For the drift term,
\[
\sum_{k=0}^{N_h-1}|a_{\nu_h+k}|
\le CN_hh,
\]
so
\[
\mathbb E\Bigl[\Bigl(\sum_{k=0}^{N_h-1}|a_{\nu_h+k}|\Bigr)^2\Bigr]
\le C(N_hh)^2.
\]
For the martingale term, Doob's $L^2$ inequality yields
\[
\mathbb E\Bigl[
\max_{0\le m\le N_h}
\Bigl|
\sum_{k=0}^{m-1}\xi_{\nu_h+k+1}
\Bigr|^2
\Bigr]
\le
4\,
\mathbb E\Bigl[
\sum_{k=0}^{N_h-1}\xi_{\nu_h+k+1}^2
\Bigr]
\le
4C\,N_hh.
\]
Therefore
\[
\mathbb E[\omega_h^2]
\le
C\bigl((N_hh)^2+N_hh\bigr)
\le
C\bigl(\theta_h^2+\theta_h+h\bigr).
\]
Consequently,
\[
\mathbb E\Bigl[
d\bigl(\widehat X^h(\tau_h+\theta_h),\widehat X^h(\tau_h)\bigr)^2
\Bigr]
\le
9\,\mathbb E[\omega_h^2]
\le
C\bigl(\theta_h^2+\theta_h+h\bigr),
\]
which tends to $0$ as $h\to0$ because $\theta_h\to0$. By Markov's
inequality,
\[
d\bigl(\widehat X^h(\tau_h+\theta_h),\widehat X^h(\tau_h)\bigr)
\to 0
\quad\text{in probability}.
\]
Together with compact containment, this
proves Aldous' criterion \cite[Theorem 8.6-(c) p.138]{ek}, and therefore the family is tight in
$D([0,T];\Gamma)$.
\end{proof}
\begin{proof}[Proof of \Cref{lem:local-generator-estimates}]
We first prove (i). Let $(\iota,x)\in\Gamma_h$ with
$x=j\sigma^\iota\sqrt h$, $j\ge1$, and set $s:=\sigma^\iota\sqrt h$,
$p_\pm:=p^\iota_\pm(x)$. Both landing points $x\pm s$ belong to
$[0,+\infty)$ (this is the crucial point: the lower one is
$(j-1)\sigma^\iota\sqrt h\ge0$, so the dynamics is never truncated and
$\varphi^\iota$ is evaluated only where it is of class $C^3_b$;
in particular, for $j=1$ the chain lands exactly at the vertex and
$\varphi^\iota(0)=\varphi(O)$). Hence
\[
\mathcal G_h\varphi(\iota,x)
=\frac1h\Bigl[
p_+\bigl(\varphi^\iota(x+s)-\varphi^\iota(x)\bigr)
+p_-\bigl(\varphi^\iota(x-s)-\varphi^\iota(x)\bigr)
\Bigr].
\]
By Taylor's formula with Lagrange remainder,
\[
\varphi^\iota(x\pm s)-\varphi^\iota(x)
=\pm s\,\partial_x\varphi^\iota(x)
+\frac{s^2}{2}\,\partial_{xx}\varphi^\iota(x)
\pm\frac{s^3}{6}\,\partial_{xxx}\varphi^\iota(\zeta_\pm),
\qquad \zeta_\pm\in[(x-s)\vee0,\,x+s].
\]
Using $p_++p_-=1$ and
$p_+-p_-=\sqrt h\,b^\iota(x)/\sigma^\iota$, we obtain
\[
\mathcal G_h\varphi(\iota,x)
=\frac{(p_+-p_-)\,s}{h}\,\partial_x\varphi^\iota(x)
+\frac{s^2}{2h}\,\partial_{xx}\varphi^\iota(x)
+O\!\left(\frac{s^3}{h}\right)
=b^\iota(x)\,\partial_x\varphi^\iota(x)
+\frac12(\sigma^\iota)^2\,\partial_{xx}\varphi^\iota(x)
+O(\sqrt h),
\]
where the constant in $O(\sqrt h)$ depends only on
$\|\partial_{xxx}\varphi\|_\infty$ and $\|\sigma\|_\infty$. Since
$\mathcal G\varphi(\iota,x)=b^\iota(x)\partial_x\varphi^\iota(x)
+\frac12(\sigma^\iota)^2\partial_{xx}\varphi^\iota(x)$, the estimate (i) follows,
uniformly over $\Gamma_h\setminus\{O\}$.
We now prove (ii). At the vertex, by the definition of the transition
mechanism,
\[
\mathcal G_h\varphi(O)
=\frac1h\,\frac{\sqrt h}{\eta+\sqrt h}
\sum_{\iota\in\mathcal I}\gamma^\iota
\bigl(\varphi^\iota(\sigma^\iota\sqrt h)-\varphi(O)\bigr)
=\frac{1}{\sqrt h\,(\eta+\sqrt h)}
\sum_{\iota\in\mathcal I}\gamma^\iota
\bigl(\varphi^\iota(\sigma^\iota\sqrt h)-\varphi(O)\bigr).
\]
Taylor's formula at $0$ on each edge gives
\[
\varphi^\iota(\sigma^\iota\sqrt h)-\varphi(O)
=\sigma^\iota\sqrt h\,\partial_x\varphi^\iota(0)
+\frac12(\sigma^\iota)^2 h\,\partial_{xx}\varphi^\iota(0)
+O(h^{3/2}),
\]
and therefore, by the vertex condition defining $D(\mathcal G)$,
\[
\mathcal G_h\varphi(O)
=\frac{\eta\,\mathcal G\varphi(O)
+\sqrt h\sum_\iota\gamma^\iota\frac12(\sigma^\iota)^2\partial_{xx}\varphi^\iota(0)
+O(h)}
{\eta+\sqrt h} .
\]
Consequently,
\[
\mathcal G_h\varphi(O)-\mathcal G\varphi(O)
=\frac{\sqrt h\,\Bigl(\sum_\iota\gamma^\iota\frac12(\sigma^\iota)^2
\partial_{xx}\varphi^\iota(0)-\mathcal G\varphi(O)\Bigr)+O(h)}
{\eta+\sqrt h},
\]
whence
$|\mathcal G_h\varphi(O)-\mathcal G\varphi(O)|
\le C\sqrt h/(\eta+\sqrt h)\le (C/\eta)\sqrt h$.
This proves (ii), and the final uniform statement follows by combining (i)
and (ii).
\end{proof}

\section{Proofs for the fully discrete HJB scheme}\label{app:proofs_hjb_scheme}

\begin{proof}[Proof of \Cref{lem:stability_hjb}]
Define $T_h:\mathcal X_h\to\mathcal X_h$ by
\[
[T_hw]_j^\iota
:=
\inf_{a\in A}\left\{
hf^\iota(x_j^\iota,a)
+
\rho_h\left[
p_+^\iota(x_j^\iota,a)w_{j+1}^\iota
+
p_-^\iota(x_j^\iota,a)w_{j-1}^\iota
\right]
\right\},
\qquad j\ge1,
\]
and
\[
[T_hw]_0
:=
\frac{\eta\theta h}{\eta+\sqrt h}
+
\rho_h\left[
\frac{\eta}{\eta+\sqrt h}w_0
+
\frac{\sqrt h}{\eta+\sqrt h}
\sum_{\kappa\in\mathcal I}\gamma^\kappa w_1^\kappa
\right].
\]
Fixed points of $T_h$ are precisely solutions of
\eqref{eq:SL_fullydiscrete_edge}--\eqref{eq:SL_fullydiscrete_vertex}.

Let $v,w\in\mathcal X_h$. Since
\[
\left|\inf_{a\in A}\alpha_a-\inf_{a\in A}\beta_a\right|
\le
\sup_{a\in A}|\alpha_a-\beta_a|,
\]
and since all transition coefficients are nonnegative and sum to one, we have
for $j\ge1$
\[
|[T_hv]_j^\iota-[T_hw]_j^\iota|
\le
\rho_h\|v-w\|_\infty.
\]
At the vertex,
\[
|[T_hv]_0-[T_hw]_0|
\le
\rho_h\left[
\frac{\eta}{\eta+\sqrt h}|v_0-w_0|
+
\frac{\sqrt h}{\eta+\sqrt h}
\sum_{\kappa\in\mathcal I}\gamma^\kappa|v_1^\kappa-w_1^\kappa|
\right]
\le
\rho_h\|v-w\|_\infty.
\]
Thus
\[
\|T_hv-T_hw\|_\infty\le\rho_h\|v-w\|_\infty.
\]
Because $0<h<\lambda^{-1}$, $0<\rho_h<1$, and $T_h$ is a contraction on the
Banach space $\mathcal X_h$. Hence it admits a unique fixed point.

For the bound, let $u_h=T_hu_h$. From the edge equation,
\[
|u_{h,j}^\iota|
\le
h\|f\|_\infty+\rho_h\|u_h\|_\infty,
\qquad j\ge1.
\]
From the vertex equation and
$0\le \eta h/(\eta+\sqrt h)\le h$, we get
\[
|u_{h,0}|
\le
h|\theta|+\rho_h\|u_h\|_\infty.
\]
Taking the supremum gives
\[
\|u_h\|_\infty
\le
h(\|f\|_\infty+|\theta|)+\rho_h\|u_h\|_\infty.
\]
Since $1-\rho_h=\lambda h$, this proves
\eqref{eq:stability_bound_hjb}.
\end{proof}

\begin{proof}[Proof of \Cref{lem:monotonicity_hjb}]
Let $v,w\in\mathcal X_h$ with $v\ge w$.
For $j\ge1$ and every $a\in A$,
\[
p_+^\iota(x_j^\iota,a)v_{j+1}^\iota
+
p_-^\iota(x_j^\iota,a)v_{j-1}^\iota
\ge
p_+^\iota(x_j^\iota,a)w_{j+1}^\iota
+
p_-^\iota(x_j^\iota,a)w_{j-1}^\iota,
\]
because $p_\pm^\iota\ge0$. Multiplying by $-\rho_h/h\le0$ and taking the
supremum in $a$ gives
\[
\mathcal S_h^\iota(x_j^\iota,r,v)
\le
\mathcal S_h^\iota(x_j^\iota,r,w).
\]
At the vertex,
\[
\mathcal S_h^0(r,v)-\mathcal S_h^0(r,w)
=
-\frac{\rho_h}{\sqrt h}
\sum_{\kappa\in\mathcal I}\gamma^\kappa
(v_1^\kappa-w_1^\kappa)
\le0.
\]
This proves the claim.
\end{proof}

\begin{proof}[Proof of \Cref{lem:consistency_hjb}]
We prove \eqref{eq:consistency_hjb_sub}. The proof of
\eqref{eq:consistency_hjb_super} is identical, with $\limsup$ replaced by
$\liminf$ and the final maximum by the corresponding minimum.

Let $h_n\to0$, $\xi_n\to0$, and $y_n\in\Gamma_{h_n}$ with $y_n\to x$.
Set $\rho_n:=1-\lambda h_n$. Passing to subsequences if necessary, we
consider the following cases.

\medskip
\noindent Case 1: $x=(\iota,s)$ with $s>0$.
For $n$ large enough, $y_n=(\iota,y_n^\iota)$ lies on the same edge and
$y_n^\iota>0$. Moreover, if $s_n:=\sigma^\iota\sqrt{h_n}$, then
$y_n^\iota-s_n>0$ for $n$ large enough. Therefore both neighboring points
belong to the physical edge. By cancellation of the constant perturbation
$\xi_n$ in the difference term,
\begin{align*}
&\mathcal S_{h_n}\bigl(y_n,\varphi(y_n)+\xi_n,\varphi_{h_n}+\xi_n\bigr)
\\
&=
\lambda(\varphi^\iota(y_n^\iota)+\xi_n)
+
\sup_{a\in A}\left\{
-\frac{\rho_n}{h_n}
\left[
p_+^\iota(y_n^\iota,a)\varphi^\iota(y_n^\iota+s_n)
+
p_-^\iota(y_n^\iota,a)\varphi^\iota(y_n^\iota-s_n)
-
\varphi^\iota(y_n^\iota)
\right]
-f^\iota(y_n^\iota,a)
\right\}.
\end{align*}
Taylor's formula gives, uniformly in $a\in A$,
\begin{align*}
&p_+^\iota(y_n^\iota,a)\varphi^\iota(y_n^\iota+s_n)
+
p_-^\iota(y_n^\iota,a)\varphi^\iota(y_n^\iota-s_n)
-
\varphi^\iota(y_n^\iota)
\\
&=
\bigl(p_+^\iota(y_n^\iota,a)-p_-^\iota(y_n^\iota,a)\bigr)s_n
\partial_x\varphi^\iota(y_n^\iota)
+
\frac{s_n^2}{2}\partial_{xx}\varphi^\iota(y_n^\iota)
+o(h_n)
\\
&=
h_n b^\iota(y_n^\iota,a)\partial_x\varphi^\iota(y_n^\iota)
+
\frac12(\sigma^\iota)^2 h_n\partial_{xx}\varphi^\iota(y_n^\iota)
+o(h_n).
\end{align*}
Consequently,
\[
\mathcal S_{h_n}\bigl(y_n,\varphi(y_n)+\xi_n,\varphi_{h_n}+\xi_n\bigr)
=
F^\iota\bigl(y_n^\iota,\varphi^\iota(y_n^\iota),
\partial_x\varphi^\iota(y_n^\iota),
\partial_{xx}\varphi^\iota(y_n^\iota)\bigr)+o(1),
\]
and the limit is the desired interior value of $F^*$.

\medskip
\noindent Case 2: $x=O$ and $y_n=O$ for infinitely many $n$.
Along this subsequence the vertex branch applies. Since the perturbation
$\xi_n$ cancels in $w_1^\kappa-r$, we have
\begin{align*}
&\mathcal S_{h_n}\bigl(O,\varphi(O)+\xi_n,\varphi_{h_n}+\xi_n\bigr)
\\
&=
-\frac{\rho_n}{\sqrt{h_n}}
\sum_{\kappa\in\mathcal I}\gamma^\kappa
\left[
\varphi^\kappa(\sigma^\kappa\sqrt{h_n})-\varphi(O)
\right]
+
\eta(\lambda(\varphi(O)+\xi_n)-\theta)
+
\lambda\sqrt{h_n}(\varphi(O)+\xi_n).
\end{align*}
Moreover,
\[
\frac{\varphi^\kappa(\sigma^\kappa\sqrt{h_n})-\varphi(O)}{\sqrt{h_n}}
=
\sigma^\kappa\partial_x\varphi^\kappa(0)+O(\sqrt{h_n}).
\]
Therefore the limit along this subsequence is
\[
-\sum_{\kappa\in\mathcal I}\gamma^\kappa\sigma^\kappa\partial_x\varphi^\kappa(0)
+
\eta(\lambda\varphi(O)-\theta)
=
K\bigl(\varphi(O),(\partial_x\varphi^\kappa(0))_{\kappa\in\mathcal I}\bigr).
\]

\medskip
\noindent Case 3: $x=O$ and $y_n\neq O$ for infinitely many $n$.
Up to a subsequence, $y_n$ belongs to a fixed edge $\Gamma^\iota$:
$y_n=(\iota,j_n\sigma^\iota\sqrt{h_n})$ with $j_n\ge1$. Set
$s_n:=\sigma^\iota\sqrt{h_n}$. Then
$y_n^\iota-s_n=(j_n-1)\sigma^\iota\sqrt{h_n}\ge0$. Thus the lower point is
still on the physical edge, and no extension through the vertex is used.
Taylor's formula up to the boundary gives, uniformly in $a\in A$,
\begin{align*}
&p_+^\iota(y_n^\iota,a)\varphi^\iota(y_n^\iota+s_n)
+
p_-^\iota(y_n^\iota,a)\varphi^\iota(y_n^\iota-s_n)
-
\varphi^\iota(y_n^\iota)
\\
&=
h_n b^\iota(y_n^\iota,a)\partial_x\varphi^\iota(y_n^\iota)
+
\frac12(\sigma^\iota)^2 h_n\partial_{xx}\varphi^\iota(y_n^\iota)
+o(h_n).
\end{align*}
Hence
\[
\mathcal S_{h_n}\bigl(y_n,\varphi(y_n)+\xi_n,\varphi_{h_n}+\xi_n\bigr)
=
F^\iota\bigl(y_n^\iota,\varphi^\iota(y_n^\iota),
\partial_x\varphi^\iota(y_n^\iota),
\partial_{xx}\varphi^\iota(y_n^\iota)\bigr)+o(1),
\]
and the limit is
\[
F^\iota\bigl(0,\varphi(O),\partial_x\varphi^\iota(0),
\partial_{xx}\varphi^\iota(0)\bigr).
\]

Combining the three cases, every subsequential limit at $O$ is either the
vertex operator $K$ or one of the edge operators
$F^\iota(0,\cdot,\cdot,\cdot)$. Taking the upper relaxed limit gives exactly
\[
F^*\bigl(O,\varphi(O),
(\partial_x\varphi^\iota(0))_{\iota\in\mathcal I},
(\partial_{xx}\varphi^\iota(0))_{\iota\in\mathcal I}\bigr).
\]
This proves \eqref{eq:consistency_hjb_sub}. The lower relaxed inequality
\eqref{eq:consistency_hjb_super} follows in the same way, taking the minimum
over the possible vertex and edge limits.
\end{proof}


\end{document}